\documentclass{siamart1116}
\usepackage{indentfirst}
\usepackage{graphicx}
\usepackage{epstopdf}
\usepackage{epsfig}
\usepackage{overpic}
\usepackage{overcite}
\usepackage{array}
\usepackage{color}
\usepackage{multirow}
\usepackage{float}
\usepackage{subfigure} 
\usepackage{amsmath,amssymb}
\usepackage[mathscr]{euscript}
\usepackage{latexsym,bm}
\usepackage{cite}
\usepackage{booktabs}
\usepackage{bbm}
\usepackage{diagbox}
 \usepackage[all]{xy}
 \usepackage{color}
 
\newtheorem{remark}{Remark}

\newcommand\D{\textup{d}}
\def\dif{\mathrm{d}}
\def\mi{\mathbbm{i}}
\def\me{\mathbbm{e}}
\def\mone{\mathbbm{1}}
\def\bx{\bm{x}}
\def\bz{\bm{z}}
\def\bk{\bm{k}}
\def\by{\bm{y}}

\def\fG{\mathcal{G}}

\makeatletter
\@addtoreset{equation}{section}
\makeatother
\renewcommand{\theequation}{\arabic{section}.\arabic{equation}}

\begin{document}
\bibliographystyle{unsrt}

\title{Branching random walk solutions to the Wigner equation}
\author{Sihong Shao\footnotemark[2] \and
Yunfeng Xiong\footnotemark[2]}
\renewcommand{\thefootnote}{\fnsymbol{footnote}}
\footnotetext[2]{LMAM and School of Mathematical Sciences, Peking University, Beijing 100871, China. Email addresses: {\tt sihong@math.pku.edu.cn} (S. Shao), {\tt xiongyf1990@pku.edu.cn} (Y. Xiong).}
\date{\today}
\maketitle



\begin{abstract}
The stochastic solutions to the Wigner equation, which explain the nonlocal oscillatory integral operator $\Theta_V$ with an anti-symmetric kernel as {the generator of two branches of jump processes},  are analyzed. All existing branching random walk solutions are formulated 
based on the Hahn-Jordan decomposition $\Theta_V=\Theta^+_V-\Theta^-_V$, i.e., treating $\Theta_V$ as the difference of two positive operators $\Theta^\pm_V$, each of which characterizes the transition of states for one branch of particles. 
Despite the fact that the first moments of such models solve the Wigner equation, we prove that the bounds of corresponding variances grow exponentially in time with the rate depending on the upper bound of $\Theta^\pm_V$, instead of $\Theta_V$. In other words, the decay of high-frequency components is totally ignored, resulting in a severe {numerical sign problem}.
{To fully utilize such decay property}, we have recourse to the stationary phase approximation for $\Theta_V$, which captures essential contributions from the stationary phase points as well as 
the near-cancelation of positive and negative weights.
The resulting branching random walk solutions are then proved to asymptotically solve the Wigner equation, but {gain} a substantial reduction in variances, thereby ameliorating the sign problem. 
Numerical experiments in 4-D phase space validate our theoretical findings.  



\vspace*{4mm}





\end{abstract}

\section{Introduction}
\label{sec:intro}


We are intended to discuss the probabilistic interpretation of the backward Wigner equation \cite{Wigner1932, CarruthersZachariasen1983, ShaoXiong2019}, arising from the recently developed particle-based simulation of the Wigner quantum dynamics \cite{NedjalkovSchwahaSelberherr2013, Wagner2016, MuscatoWagner2016, ShaoXiong2019}. The backward Wigner equation is a partial integro-differential equation defined in phase space $(\bx, \bk) \in \mathbb{R}^n \times \mathbb{R}^n$ with an ``initial'' condition $\varphi_T \in L^2(\mathbb{R}^n \times \mathbb{R}^n)$.  
\begin{align}
&\frac{\partial }{\partial t}\varphi(\bx, \bk, t) + \frac{\hbar \bk}{m} \cdot \nabla_{\bx} \varphi(\bx, \bk, t) = \Theta_{V}[\varphi](\bx, \bk, t), ~~0 \le t\le T, \label{eq.backward_Wigner}\\
& \varphi(\bx, \bk, T) = \varphi_T(\bx, \bk),
\end{align}
Here $\varphi(\bx, \bk,t)$ is the dual Wigner function, $m$ is the mass, $\hbar$ represents the reduced Planck constant and the pseudo-differential operator (PDO) $\Theta_{V}$  reads 
\begin{equation}\label{def.pseudo_differential_operator}
\Theta_V[\varphi](\bx, \bk, t) = \frac{1}{\mi \hbar (2\pi)^n}\int_{\mathbb{R}^n \times \mathbb{R}^n} \me^{\mi(\bk-\bk^{\prime})\cdot \by} D_V(\bx, \by, t) \varphi(\bx, \bk^{\prime}, t) \D \by \D \bk^{\prime},
\end{equation}
with $D_V(\bx, \by, t) = V(\bx-\by/{2}, t) - V(\bx + \by/{2}, t)$ (i.e., the central difference of the external potential $V(\bx, t)$). Obviously, $D_V(\bx, \by, t)$ is anti-symmetric in {$\by$-variable},
\begin{equation}\label{eq:anti}
D_V(\bx, \by, t) = -D_V(\bx, -\by, t).
\end{equation}
It is well known that $\Theta_V$, a nonlocal operator with an anti-symmetric symbol, actually characterizes a deformation of the classical Poisson bracket \cite{Kontsevich2003} and exactly reflects the nonlocal nature of quantum mechanics \cite{tatarskiui1983,WeinbubFerry2018,ChenXiongShao2019,ChenShaoCai2019}.  

The subsequent analysis will be based on two equivalent representations of the PDO. The first form is the kernel {representation:} 
\begin{align}
\Theta_V[\varphi](\bx, \bk, t) & = \int_{\mathbb{R}^n} V_W(\bx, \bk-\bk^{\prime}, t) \varphi(\bx, \bk^{\prime}, t) \D \bk^{\prime} \label{def.first_form_PDO},
\end{align}
with 
the real-valued kernel function $V_W$ (termed the Wigner kernel)  
\begin{equation}\label{def.Wigner_kernel}
\begin{split}
V_W(\bx, \bk, t) &= \frac{1}{\mi \hbar (2 \pi)^n} \int_{\mathbb{R}^n}  (V(\bx-\frac{\by}{2}, t) - V(x + \frac{\by}{2}, t)) \me^{\mi \bk \cdot \by} \D \by \\
& = \frac{1}{\mi \hbar \pi^n}\mathcal{F}_{\bx \to \bk}V(2\bk, t) \me^{2 \mi \bk \cdot \bx} -  \frac{1}{\mi \hbar \pi^n}\mathcal{F}_{\bx \to \bk}V(-2\bk, t) \me^{-2\mi \bk\cdot \bx} \\
& =  {2^n} \psi(2\bk, t) \me^{2\mi \bk \cdot \bm{z(\bx})} - {2^n} \psi(-2\bk, t) \me^{-2\mi \bk \cdot \bz(\bx)}.
\end{split}
\end{equation}
Here $\mathcal{F}_{\bx \to \bk}V(\bk,t)=\int_{\mathbb{R}^n} V(\bx,t) \me^{-\mi\bk\cdot\bx}\dif \bx$ denotes the Fourier transform of the potential function $V(\bx, t)$ in $\bx$-variable, and 
\begin{equation}
\psi(\bk, t) = \frac{1}{\mi \hbar (2\pi)^n}\me^{\mi \bk \cdot (\bx -\bz(\bx))} \mathcal{F}_{\bx \to \bk}V(\bk, t).
\end{equation}
It is realized that the kernel is anti-symmetric in $\bk$-variable
\begin{equation}\label{eq.anti_symmetry}
V_W(\bx, \bk -\bk^{\prime} , t) = -V_W(\bx, \bk^{\prime}-\bk, t).
\end{equation} 
due to the anti-symmetry of $D_V$ (see Eq.~\eqref{eq:anti}). 

The second form is the oscillatory integral representation: 
\begin{equation}\label{def.second_form_PDO}
\Theta_V[\varphi](\bx, \bk, t) = \int_{\mathbb{R}^n} \me^{\mi \bz(\bx) \cdot \bk^{\prime}} \psi(\bk^{\prime}, t)(\varphi(\bx, \bk - \frac{\bk^{\prime}}{2}, t) - \varphi(\bx, \bk + \frac{\bk^{\prime}}{2}, t)) \D \bk^{\prime},
\end{equation}
which facilitates the derivation of its asymptotic expansion (see Theorem \ref{thm.stationary_phase_approximation}).

In order to extend $\Theta_V$ to a bounded operator from $L^2(\mathbb{R}^{2n})$ to itself, say, there exists a uniform upper bound $K_V$ such that
\begin{equation}\label{eq.PDO_bound}
\Vert \Theta_V[\varphi](t) \Vert_2 \le K_V \Vert \varphi(t) \Vert_2,
\end{equation}
{we make the following assumptions for a finite time interval $ [0, T]$}.

\begin{itemize}

\item[\textbf{(A1):}] $\varphi \in C([0, T], L^2(\mathbb{R}^n \times \mathbb{R}^n))$ and is localized in $(\bx, \bk)$-space for any $t\in [0, T]$, with the {\it minimal} compact support denoted by $\mathcal{X} \times \mathcal{K} \subset \mathbb{R}^n \times \mathbb{R}^n$.

\item[\textbf{(A2):}] Suppose either of the following conditions holds: 
\begin{itemize}

\item[(1)] $\psi \in C([0, T],  C^\infty(\mathbb{R}^n) \cap L^1(\mathbb{R}^n)  )$;

\item[(2)] $\psi \in C([0, T], C^{\infty}(\mathbb{R}^n \setminus \{0\}) \cap L^1_{loc}(\mathbb{R}^n))$,
\end{itemize}
and there exist a radial function $\Psi(|\bk|) \in L^1_{loc}(\mathbb{R}^n)$, such that $|\psi(\bk, t)| \le \Psi(|\bk|)$ in $\mathbb{R}^n \setminus \{0\}$ and {$\Psi(|\bk|) \le C_{n, \alpha} |\bk|^{-n+\alpha}$ holds} for sufficiently large $|\bm{k}|$ and {given constants $C_{n, \alpha}$ and $\alpha \in (0, n)$};

\end{itemize}
{Here $\Vert \cdot \Vert_ p$ is short for $L^p_{\bx} \times L^p_{\bk}$ norm, say, $\Vert \varphi(t) \Vert_{L^p_{\bx} \times L^p_{\bk}}= (\int_{\mathbb{R}^{n} \times \mathbb{R}^{n}}  | \varphi(\bx, \bk, t)  |^p \D \bx \D \bk)^{{1}/{p}}$.}

The prototypes for the latter condition in {\bf (A2)} arise from quantum molecular systems and fractional diffusion problems \cite{DuGunzburgerLehoucqZhou2012, Toniazzi2018}. When the potential is of the Coulomb type $V(\bx) = |\bx - \bx_A|^{-1}$, it is easy to verify that $\psi(\bk) \propto |\bk|^{-n+1}$ and $\bz(\bx) = \bx -\bx_A$, so that the symbol functions may have singularities at $\bk = 0$ and $\bk = \infty$. Therefore, we need to {focus on the weakly singular convolution \cite{bk:Kress2014}}, instead of solely treating it in the classical symbol class $C([0, T], S^0(\mathbb{R}^n \times \mathbb{R}^n))$. 


Now we turn to the probabilistic perspective. The starting point of the stochastic solution is to cast Eq.~\eqref{eq.backward_Wigner} into its equivalent integral formulation by adding a term $-\gamma_0 \cdot \varphi(\bx, \bk, t)$ on both sides of Eq.~\eqref{eq.backward_Wigner} \cite{ShaoXiong2019},
\begin{equation}\label{def.backward_renewal_type_equation}
\begin{split}
\varphi(\bx, \bk, t)=&(1-\mathcal{G}(T-t)) \varphi_T(\bx(T-t), k) +\int_t^T \D \mathcal{G}(t^{\prime}-t) \\
& \times \int_{\mathbb{R}^n} (-\frac{V_W(\bx(t^{\prime}-t), \bk^{\prime}, t^{\prime})}{\gamma_0}+\delta(\bk^{\prime}))\varphi(\bx(t^{\prime}-t), \bk - \bk^{\prime}, t^{\prime}) \D \bk ^{\prime},
\end{split}
\end{equation}
the derivation of which will be put in Section~\ref{sec:pre}.
{The constant} parameter $\gamma_0$ turns out to be the intensity of an exponential distribution as follows, 
\begin{equation}\label{eq:measure2}
\mathcal{G}(t^{\prime}-t)= 1- \me^{-\gamma_0 (t^{\prime}-t)}, \quad \D  \mathcal{G}(t^{\prime}-t) = \gamma_0 \me^{-\gamma_0 (t^{\prime}-t)}, \quad t^{\prime} \ge t.
\end{equation}

The main problem is how to resolve the negative values of kernel $V_W$. In constrast to nonlocal operators with nonnegative and symmetric kernels \cite{DuGunzburgerLehoucqZhou2012, DuToniazziZhou2018, Toniazzi2018}, the existing stochastic approach is based on the unique Hahn-Jordan decomposition  (HJD) \cite{bk:Kallenberg2002}:
\begin{align}
\Theta_{V}[\varphi](\bx, \bk, t)  &= \Theta^+_{V}[\varphi](\bx, \bk, t)  - \Theta^-_{V}[\varphi](\bx, \bk, t) , \label{eq.PDO_splitting} \\
\Theta_V^{\pm}[\varphi](\bx, \bk, t) &= \int_{\mathbb{R}^n} V_W^\pm(\bx, \bk-\bk^{\prime}, t) \varphi(\bx, \bk^{\prime}, t) \D \bk^{\prime},  \\
V_W^{\pm}(\bx, \bk, t)  &= \max\{ \pm V_W(\bx, \bk, t), 0 \},
\end{align}
so that $V_W^\pm \in C([0, T], L^1_{loc}(\mathbb{R}^{n}\times \mathbb{R}^n))$ become positive semi-definite kernels. Moreover, {we assume that there exists a uniform normalizing bound $\breve{\xi}$ for $(\bx, \bk) \in \mathcal{X} \times 2\mathcal{K}$}, 
\begin{equation}\label{eq.PDO_plus_bound}
 \gamma_0 \ge \breve{\xi} = \max_{0 \le t \le T} \max_{x \in \mathcal{X}} \int_{\mathbb{R}^n } V_W^\pm (\bx, \bk, t) \mone_{\{ \bk \in 2\mathcal{K} \}}  \D \bk.
\end{equation}
 
It follows that the probabilistic interpretation is to seek a branching random walk model (BRW) such that its first moment satisfies the renewal-type Wigner (W) equation \eqref{def.backward_renewal_type_equation}  \cite{SellierNedjalkovDimov2014, Wagner2016, ShaoXiong2019}, dubbed WBRW-HJD hereafter.  Such model describes a mass distribution of a random cloud starting at $Q= (\bx, \bk)$ and frozen at random states and exhibiting both random motion and random growth. The random variable is a family history $\Omega$, a denumerable random sequence corresponding to a unique family tree \cite{bk:Harris1963}, and $\mathscr{B}_{\Omega}$ is the Borel extension of cylinder sets on $\Omega$. 
The particles in the family history $\Omega$ move according to the following five rules.
\begin{itemize}

\item[(1)](Markov property) The motion of each particle is described by a {right continuous Markov process}.

\item[(2)](Memoryless life-length) The particle at $(\bx, \bk, t)$ dies in the age time interval $(t, t+\tau) $  with probability $1 - \me^{-\gamma_0 \tau}$.

\item[(3)](Frozen state) The particle at $(\bx, \bk,t)$ is frozen at the state $(\bx(T-t), \bk)$ when its life-length $\tau \ge T-t$.

\item[(4)](Branching property) The particles at $(\bx, \bk, t)$, carrying a weight $w$, dies at age $t+\tau$ at state $(\bx(\tau), \bk)$ when $\tau < T - t$,  and produces at most five new offsprings at  states $(\bx_{(1)}, \bk_{(1)})$,  $(\bx_{(2)}, \bk_{(2)}) \cdots (\bx_{(5)}, \bk_{(5)})$, endowed with updated weights $w_{(1)}$, $w_{(2)}, \cdots, w_{(5)}$, respectively.  \label{def.Branching_property}

\item[(5)](Independence) The only interaction between the particles is that the birth time and state of offsprings coincide with the death time and state of their parent.

\end{itemize}

We are able to define a probability measure on the measurable space $(\Omega, \mathscr{B}_{\Omega})$ and thus the stochastic process based on a specific setting of the transition kernels and particle weights in the fourth rule (vide post). Roughly speaking, WBRW-HJD can be categorized into the weighted-particle (wp) \cite{ShaoXiong2019} and signed-particle (sp) \cite{SellierNedjalkovDimov2014,Wagner2016,XiongShao2019} implementations, denoted by $\mathrm{X}_t^\mathrm{w}$ and $\mathrm{X}_t^\mathrm{s}$ associated with the probability laws $\Pi_Q^\mathrm{w}$ and $\Pi_Q^\mathrm{s}$, respectively. 
It has been shown in \cite{ShaoXiong2019} that (taking $\mathrm{X}_t^\mathrm{w}$ as an example), 
\begin{equation}\label{eq.expection_X_t}
\Pi_Q^\mathrm{w} \mathrm{X}_t^\mathrm{w} =  \varphi(\bx, \bk, t) 
\end{equation}
holds on some kind of probability space $(\Omega, \mathscr{B}_{\Omega},\Pi_Q^\mathrm{w})$,
where $\Pi_Q^\mathrm{w} \mathrm{X}_t^\mathrm{w}$ means the expectation of $\mathrm{X}_t^\mathrm{w}$ with respect to the probability measure $\Pi_Q^\mathrm{w}$. However, {to the best of} our knowledge, the related variance estimation has not been established. To this end, our first contribution is to estimate the variance of WBRW-HJD, as stated in Theorem \ref{thm.variance_estimate}. 

\begin{theorem}[Variance of WBRW-HJD]
\label{thm.variance_estimate}
Suppose {\bf(A1)} and {\bf (A2)} are satisfied and let $\gamma_1 = 2K_V\gamma_0+2\breve{\xi}^2$. Then the variances of $\mathrm{X}_t^\mathrm{w}$ and $\mathrm{X}_t^\mathrm{s}$ satisfy
\begin{align}
\Vert \Pi_{Q}^\mathrm{w} (\mathrm{X}_t^\mathrm{w} - \varphi(t))^2 \Vert_1 &\le (1+\frac{\gamma_1}{\gamma_0}(T-t) )\me^{2\max(K_V, \frac{\breve{\xi}^2}{\gamma_0}) (T-t)} \Vert \varphi_T \Vert_2^2 - \Vert \varphi(t) \Vert_2^2, \label{eq.wbrw_wp_variance_estimate}\\
\Vert  \Pi_{Q}^\mathrm{s} (\mathrm{X}_t^\mathrm{s} - \varphi(t))^2\Vert_1 &\le (1+\frac{\gamma_1}{\gamma_0}(T-t) ) \me^{2\breve{\xi}(T-t)} \Vert \varphi_T \Vert_2^2 - \Vert \varphi(t) \Vert_2^2. \label{eq.wbrw_sp_variance_estimate}
\end{align}
\end{theorem}

Two key observations are readily seen from Theorem \ref{thm.variance_estimate}. 
One is the exponential rate for spWBRW-HJD is $2\breve{\xi}$,  
{that} depends on the volume of the support $\mathcal{K}$ and thus cannot be improved. {This poses a huge challenge for high dimensional problems since $\breve{\xi}$ usually depends on $n$ exponentially.} By contrast,  the rate for wpWBRW-HJD can be reduced by increasing $\gamma_0$, and the optimal exponential rate is $2K_V$.
Definitely, $K_V$ is usually far less than $\breve{\xi}$, 
implied by Eqs.~\eqref{eq.PDO_bound} and \eqref{eq.PDO_plus_bound}. In this sense, the latter outperforms the former. The other is the large exponential rates $2\breve{\xi}^2/\gamma_0$ and $2\breve{\xi}$, {introduced by HJD \eqref{eq.PDO_splitting}, lead to a rapid growth of variance. Such phenomenon is called ``numerical sign problem''
\cite{LohJrGubernatis1990} as the Hahn-Jordan decomposition of a signed measure totally ignores the near-cancellation of positive and negative weights}. 

Our second contribution is to formulate a new class of BRW solutions, dubbed WBRW-SPA, to diminish the variance growth. The motivation comes from the stationary phase method, a useful technique in microlocal analysis \cite{bk:Sogge2014-AM188}, which makes full use of the essential contribution from the localized parts (see Theorem 2). As a consequence, the upper bounds in Eqs.~\eqref{eq.wbrw_wp_variance_estimate} and \eqref{eq.wbrw_sp_variance_estimate} can be significantly reduced especially in the region where the module $|\bz(\bx)|$ is sufficiently large (see Theorem \ref{SPA_thm.variance_estimate}). 




\begin{theorem}[Stationary phase approximation]\label{thm.stationary_phase_approximation}
Suppose $|\bz(\bx)| \ne 0$ and the amplitude function $\psi \in C([0, T], C_0^{\infty}(\mathbb{R}^n \setminus \{0\}) \cap L^1_{loc}(\mathbb{R}^n))$. 
Then for a sufficiently large $\lambda_0$, we have a stationary phase approximation $\Theta^{\lambda_0}_V[\varphi]$  to PDO
\begin{align}
\Theta_V[\varphi](\bx, \bk, t) &= \Theta_V^{\lambda_0}[\varphi](\bx, \bk, t) + \mathcal{O}(\lambda_0^{-{n}/{2}}), \\
\Theta^{\lambda_0}_V[\varphi](\bx, \bk, t) &= \Lambda^{< \lambda_0}[\varphi] (\bx, \bk,t) + \Lambda_+^{>\lambda_0}[\varphi] (\bx, \bk,t) + \Lambda_-^{>\lambda_0}[\varphi] (\bx, \bk,t), \label{def.stationary_phase_approximation}
\end{align}
where 
\begin{align*}
 \Lambda^{<\lambda_0}[\varphi] (\bx, \bk,t) &= \int_{B(\frac{\lambda_0}{|\bz(\bx)|})} \me^{\mi \bz(\bx) \cdot \bk^{\prime}} \psi(\bk^{\prime}, t) \Delta_{\bk^{\prime}}[\varphi](\bx, \bk, t) \D \bk^{\prime}, \\
\Lambda_\pm^{>\lambda_0}[\varphi] (\bx, \bk, t) &= \int_{\frac{\lambda_0}{|\bz(\bx)|}}^{+\infty} \me^{\pm\mi r |\bz(\bx)|}\left(\frac{ 2\pi}{\pm \mi r |\bz(\bx)|}\right)^{\frac{n-1}{2}} r^{n-1}{\psi}(r\sigma_\pm, t) \Delta_{r\sigma_\pm} [\varphi](\bx, \bk, t) \D r,
\end{align*}
in the sense that there exists a positive constant $C$, which depends on $\psi$ and its first derivate but is independent on $\lambda_0$,  such that
\begin{equation}\label{eq.error_estimate_stationary_phase_method}
\Vert \Theta_V[\varphi](t) - \Theta_V^{\lambda_0}[\varphi](t) \Vert_2 \le C \lambda_0^{-{n}/{2}} \Vert \varphi(t) \Vert_{L^2_{\bx} \times H^1_{\bk}}.
\end{equation}
Here the norm is $\Vert \varphi(t) \Vert_{L^2_{\bx} \times H^1_{\bk}} = \Vert \varphi(t) \Vert_2 +  \Vert \nabla_{\bk}\varphi(t) \Vert_2$ and $B(r)$ is a closed ball with radius $r$ centered at the origin, $\sigma_\pm$ (short for $\sigma_\pm(\bx))$ represent two critical points on the $(n-1)$-dimensional unit  spherical surface with normal vectors pointing in (or opposite to) the direction of $\bz(\bx)=(z_1,z_2,\ldots,z_n)$,
which can be parameterized by  
\begin{equation*}
\sigma_{\pm}= (\cos\vartheta_1^\pm, \sin\vartheta_1^\pm \cos\vartheta_2^\pm, \ldots, \sin\vartheta_1^\pm\cdots
\sin\vartheta_{n-2}^\pm\cos\vartheta_{n-1}^\pm, \sin\vartheta_1^\pm\cdots \sin\vartheta_{n-1}^\pm),
\end{equation*}
with
\begin{equation}\label{def.critical_point} 
\begin{split}
\vartheta_i^\pm &= \textup{arccot}(\pm z_i/\sqrt{z_{i+1}^2 +\cdots +z_n^2}) \in [0,\pi], ~~ i=1,2,\ldots,  n-2, \\
\vartheta_{n-1}^\pm &= 2\hspace{0.05cm}\textup{arccot}(\pm(z_{n-1} +\sqrt{z_{n-1}^2 + z_{n}^2})/z_n) \in [0, 2\pi),
\end{split}
\end{equation}
and $\Delta_{\bk^{\prime}}$ is the central difference operator
\begin{equation}
\Delta_{\bk^{\prime}} [\varphi](\bx, \bk, t) = \varphi(\bx, \bk - \frac{\bk^{\prime}}{2} , t) - \varphi(\bx, \bk + \frac{\bk^{\prime}}{2} , t).
\end{equation}
\end{theorem}

Intuitively speaking, the parameter $\lambda_0$ serves as a filter to decompose PDO into a low-frequency component $\Lambda^{<\lambda_0}$ and a high-frequency one, the leading terms of which are $\Lambda_\pm^{>\lambda_0}$, and use the resulting nonlocal operator $\Theta_V^{\lambda_0}$ to directly formulate WBRW-SPA, instead of $\Theta_V$ as adopted in WBRW-HJD. Specifically, we still use HJD to deal with $\Lambda^{<\lambda_0}$ and tackle $\Lambda_+^{>\lambda_0}+\Lambda_-^{>\lambda_0}$ by another two branches of particles, yielding two stochastic processes:  the ``wp'' implementation $\mathrm{Y}_t^\mathrm{w}$ and the ``sp'' implementation $\mathrm{Y}_t^\mathrm{s}$, associated with the probability measures $ \mathsf{\Pi}^\mathrm{w}_{Q}$ and  $\mathsf{\Pi}^\mathrm{s}_{Q}$, respectively. In order to estimate the effect of low-frequency parts, we further need the following assumptions.
\begin{itemize}
\item[\textbf{(A3):}]  $\varphi \in C([0, T], L^2(\mathbb{R}^n) \times H^1(\mathbb{R}^n))$ and is localized in $(\bx, \bk)$-space for any $t\in [0, T]$ with the {\it minimal} compact support denoted by $\mathcal{X} \times \mathcal{K} \subset \mathbb{R}^n \times \mathbb{R}^n$;

\item[\textbf{(A4):}] For the positive constant $\breve{\xi}$ in Eq.~\eqref{eq.PDO_plus_bound} there exist positive constants $\lambda_0 > 1$ and $\alpha_\ast < 1$ such that
\begin{equation}\label{eq.VW_plus_truncated_2}
\alpha_\ast \breve{\xi} = \max_{0\le t\le T} \max_{\bx \in \mathcal{X}} \int_{\mathbb{R}^n } V_W^\pm (\bx, \bk, t) {\mone_{\{ |2\bk| < {\lambda_0}/{|\bz(\bx)|}\}} } \D \bk.
\end{equation}
\end{itemize}

The assumption {\bf (A4)} { indicates that the normalizing bound for $V_W^\pm$ can be diminished when $\bk$ is restricted in a smaller domain, which holds if $\min_{\bx \in \mathcal{X}}|\bz(\bx)|$ is large enough}. For instance, 
$V(\bx) = |\bx - \bx_A|^{-1}$, it requires the displacement {$\min_{\bx \in \mathcal{X}} |\bx - \bx_A|$} is sufficiently large. Accordingly, we are able to show that the first moment of WBRW-SPA turns out to be an asymptotic approximation to the solution of Eq.~\eqref{eq.backward_Wigner}. We also study its deviation from the dual Wigner function $\varphi$ by estimating the second moment (also termed ``variance'' hereafter) and find that, in contrast to Eqs.~\eqref{eq.wbrw_wp_variance_estimate} and \eqref{eq.wbrw_sp_variance_estimate}, the exponential growth rate in the upper bound is suppressed, so that a moderate increase of variance can be achieved.

\begin{theorem}[WBRW-SPA]\label{SPA_thm.variance_estimate}
Suppose {\bf (A2)}-{\bf (A4)} are satisfied and let $\gamma_2 = \alpha_\ast \breve{\xi}^2$. Then for a sufficient large $\lambda_0$,  there exist a weighted-particle branching random walk model $\mathrm{Y}_t^\mathrm{w}$ and a signed-particle one $\mathrm{Y}_t^\mathrm{s}$ on the probability spaces $(\Omega, \mathscr{B}_{\Omega}, \mathsf{\Pi}^\mathrm{w}_{Q})$ and $(\Omega, \mathscr{B}_{\Omega}, \mathsf{\Pi}^\mathrm{s}_{Q})$, respectively, such that
\begin{equation}\label{eq.expection_Y_t}
\mathsf{\Pi}_{Q}^\mathrm{w} \mathrm{Y}_t^\mathrm{w}= \mathsf{\Pi}_{Q}^\mathrm{s} \mathrm{Y}_t^\mathrm{s} = \varphi(\bx, \bk, t) + \mathcal{O}(\lambda_0^{-{n}/{2}}),
\end{equation}
and their variances satisfy
\begin{align} \label{eq.wbrw_spa_wp_variance_estimate}
\Vert  \mathsf{\Pi}_{Q}^\mathrm{w} (\mathrm{Y}_t^\mathrm{w} - \varphi(t))^2 \Vert_1 &\lesssim (1+\frac{4\gamma_2}{\gamma_0}(T-t))\me^{2\max(K_V, \frac{\alpha_\ast \breve{\xi}^2}{\gamma_0}) (T-t)} \Vert \varphi_T \Vert_2^2 - \Vert \varphi(t) \Vert_2^2 , \\
 \label{eq.wbrw_spa_sp_variance_estimate}
\Vert  \mathsf{\Pi}_{Q}^\mathrm{s} (\mathrm{Y}_t^\mathrm{s} - \varphi(t))^2 \Vert_1 &\lesssim (1+2(K_V+ \frac{\gamma_2}{\gamma_0})(T-t))\me^{2\alpha_\ast \breve{\xi}(T-t)} \Vert \varphi_T \Vert_2^2 - \Vert \varphi(t) \Vert_2^2.
\end{align}
\end{theorem}

The rest is organized as follows. Section \ref{sec:pre} briefly reviews the basic of the Wigner equation. Section \ref{sec:spa} derives the $L^2$-boundedness and the stationary phase approximation to PDO. WBRW-HJD and WBRW-SPA are analyzed in Sections  \ref{sec:WBRW-HJD} and \ref{sec:WBRW-SPA}, respectively. In Section \ref{sec:num}, a typical numerical experiment is performed to verify our theoretical analysis. This paper is concluded in Section \ref{sec:con}.

\section{The Wigner equation}
\label{sec:pre}

The Wigner equation, introduced by Wigner in his pioneering work  \cite{Wigner1932}, provides a fundamental phase space description of quantum mechanics, 
and quantum behavior is completely characterized by the {nonlocal} pseudo-differential operator $\Theta_V[f]$ defined in Eq.~\eqref{def.pseudo_differential_operator}. Mathematically speaking, it is a partial integro-differential equation defined in phase space $(\bx, \bk) \in \mathbb{R}^n \times \mathbb{R}^n$
\begin{equation}
\begin{split}
&\frac{\partial }{\partial t}f(\bx, \bk, t)+\frac{\hbar \bk}{m} \cdot \nabla_{x} f(\bx, \bk, t)=\Theta_{V}\left[f\right](\bx, \bk, t), ~~0 \le t \le T, \\
& f(\bx, \bk, 0) = f_0(\bx, \bk),
\end{split}
\label{eq.Wigner}
\end{equation}
with an initial value $f_0 \in L^2(\mathbb{R}^n \times \mathbb{R}^n)$. 
The weak formulation of the Wigner equation is of great importance since any quantum observable $\langle \hat{A} \rangle(t)$ can be expressed by its Weyl symbol $A_W(\bx, \bk)$ averaged by the Wigner function \cite{tatarskiui1983}, namely, $\langle \hat{A} \rangle(t) = \langle A_W, f(t) \rangle$ with 
\begin{equation}
\langle f(t), \varphi(t) \rangle  = \int_{\mathbb{R}^n \times \mathbb{R}^n}f(\bx, \bk, t) \varphi(\bx, \bk, t) \D \bx \D \bk.
\end{equation}
Thus it motivates to study the dual system \cite{CarruthersZachariasen1983} and derive the adjoint equation of Eq.~\eqref{eq.Wigner} under a non-degenerate inner product:
\begin{equation}\label{def.nd_inner_product}
\langle f, \varphi \rangle_T =  \int_0^T \langle f(t), \varphi(t) \rangle \D t = \int_{\mathbb{R}^n \times \mathbb{R}^n \times [0,T]} f(\bx, \bk, t) \varphi(\bx, \bk, t) \D \bx \D \bk \D t,
\end{equation}
where $T$ is a fixed time instant and $\varphi \in C([0,T],  L^2(\mathbb{R}^n \times \mathbb{R}^n))$ is a test function with a compact support in $\mathbb{R}^n \times \mathbb{R}^n$.
Using the anti-symmetry \eqref{eq.anti_symmetry} of the Wigner kernel, we have  
\begin{equation}
\langle \Theta_V[f], \varphi \rangle = - \langle f, \Theta_V[\varphi] \rangle
\end{equation} 
and integration by parts directly leads to  
\begin{equation*}
\begin{split}
\langle \frac{\partial f}{\partial t} + \frac{\hbar \bk}{m} \cdot \nabla_{\bx} f - \Theta_{V}[f], \varphi \rangle_T & = \langle \frac{\partial f}{\partial t} , \varphi\rangle_T + \langle \frac{\hbar \bk}{m} \cdot \nabla_{\bx} f, \varphi \rangle_T - \langle \Theta_{V}[f], \varphi \rangle_T\\
&\hspace{-0.08cm} = \langle f_T, \varphi_T \rangle - \langle f_0, \varphi_0 \rangle - \langle f, \frac{\partial \varphi}{\partial t} + \frac{\hbar \bk}{m} \cdot \nabla_{\bx} \varphi - \Theta_{V}[\varphi] \rangle_T,
\end{split}
\end{equation*}
where $f_T$ and $\varphi_T$ are short for $f(\bx, \bk, T)$ and $\varphi(\bx, \bk, T)$, respectively.  Therefore the adjoint correspondence, i.e.,  the backward Wigner equation \eqref{eq.backward_Wigner}, 
is immediately derived by setting 
\begin{equation}\label{eq.averaged}
\langle \varphi_T, f_T \rangle = \langle \varphi_0,  f_0 \rangle.
\end{equation}
Formally, Eq.~\eqref{eq.averaged} allows us to evaluate the quantum mechanical observable $\langle \varphi_T, f_T\rangle$  {\sl only by the {``initial''} data} \cite{ShaoXiong2019}.

The backward Wigner equation \eqref{eq.backward_Wigner} can be cast into a renewal-type equation by adding a term $-\gamma_0 \cdot \varphi(\bx, \bk, t)$ on both sides, 
\begin{equation}
\frac{\partial}{\partial t} \varphi (\bx, \bk, t) + \frac{\hbar \bk}{m} \cdot \nabla_{\bx} \varphi(\bx, \bk, t)  - \gamma_0 \cdot \varphi(\bx, \bk, t) = \Theta_{V}[\varphi](\bx, \bk, t) - \gamma_0 \cdot \varphi(\bx, \bk, t),
\end{equation}
with $\gamma_0$ being a prescribed constant (see Eq.~\eqref{eq.PDO_plus_bound}),
and the mild solution reads
\begin{equation*}
\begin{split}
\varphi(\bx, \bk, t) &= \me^{(T-t) \mathcal{A}} \varphi_T(\bx, \bk)  - \int_t^T \me^{(t^{\prime}-t) \mathcal{A}}(\Theta_V[\varphi](\bx, \bk, t^{\prime}) - \gamma_0 \cdot \varphi(\bx, \bk, t^{\prime})) \D t^{\prime},
\end{split}
\end{equation*}
by the variation-of-constant formula \cite{bk:Pazy1983}.
Here $\me^{\Delta t \mathcal{A}}$ is short for the semigroup generated by $ \mathcal{A} = \hbar \bk/m \cdot \nabla_{\bx}  - \gamma_0$, and its action on a given function can be now readily performed.
For instance, we have 
\begin{equation}\label{forward_semi_group}
\me^{(t^{\prime}-t) \mathcal{A}} g(\bx, \bk, t^{\prime}) = \me^{-\gamma_0(t^{\prime}-t)} g(\bx(t^{\prime}-t), \bk, t^{\prime}), \quad t^{\prime} \ge t, 
\end{equation}
where $\bx(\Delta t) = \bx + \hbar \bk \Delta t/m$ gives the {\sl forward-in-time trajectory} of $(\bx, \bk)$ with a positive time increment $\Delta t$. That is, the backward renewal-type equation Eq.~\eqref{def.backward_renewal_type_equation} is thus verified.

\section{$L^2$-boundedness and stationary phase approximation}
\label{sec:spa}


Before proceeding to the probabilistic aspect, we first need to establish the $L^2$-boundedness
of $\Theta_V$ under the assumptions (\textbf{A1}) and (\textbf{A2}). 
For $\psi \in C([0, T], L^1(\mathbb{R}^n))$, Eq.~\eqref{eq.PDO_bound} is readily verified by Young's convolution inequality, whereas the $L^2$-boundedness for weakly singular kernels is obtained {by the Hardy-Littlewood-Sobolev theorem} \cite{bk:LuDingYan2007}. 
After that, we present the stationary phase approximation and detail its remainder estimate. 

Suppose $(\textbf{A1})$ and the second condition of $(\textbf{A2})$ hold, then $L^p(\mathbb{R}^{2n}) \subset L^2(\mathbb{R}^{2n})$ for $1\le p < 2$ due to the H\"{o}lder's inequality:
\begin{equation}\label{eq.L2_embedding}
\Vert \varphi(t) \Vert_p = \Vert \varphi(t) \cdot \mone_{\mathcal{X}} \cdot \mone_{\mathcal{K}}\Vert_p \le \Vert \mone_{\mathcal{X}} \cdot \mone_{\mathcal{K}} \Vert_{\frac{2p}{2-p}} \Vert \varphi(t) \Vert_2 < \infty
\end{equation}
as $\mathcal{X} \times \mathcal{K}$ has a finite measure. Next we introduce a smooth cut-off function $\chi_{\varepsilon, R}(r) \in C^\infty([0, +\infty))$:
\begin{equation}\label{def.smooth_function}
\chi_{\varepsilon, R}(r) = \left\{
\begin{split}
&1, \quad r \in [\varepsilon, 2R], \\
&0, \quad r \in [0, \varepsilon/2)\cup(3R,+\infty),
\end{split}
\right.
\end{equation}
and let $\psi_{\varepsilon} = \psi \cdot \chi_{\varepsilon, R}(|\bk|)$, $\psi_{\infty} = \psi \cdot (1-\chi_{\varepsilon, R}(|\bk|)) \cdot \mone_{\{|\bk| \ge 2R\}}$. Here {$\varepsilon$ is introduced to remove the singularity at $\bk = 0$} and
$R$ is chosen sufficient large to ensure {$\Psi(|\bk|) \le C_{n, \alpha}|\bk|^{-n+\alpha}$, as stated in assumption ${\bf (A2)}$}. Then it is readily verified that the truncated operator $\Theta^\varepsilon_V[\varphi]$ has the following estimate 
\begin{equation}\label{def.pdo_L2_bound}
\begin{split}
\Vert \Theta_V^\varepsilon[\varphi](t) \Vert_{L^2_{\bk}} & \le  2^{n+1}\Vert \int_{\mathbb{R}^n} \me^{2\mi(\bk - \bk^{\prime}) \cdot \bz(\bx)} (\psi_\varepsilon + \psi_\infty)(2(\bk - \bk^{\prime}), t) \varphi(\bx, \bk^{\prime}, t) \D \bk^{\prime}  \Vert_{L^2_{\bk}} \\
&\le  2^{n+1}\Vert \int_{\mathbb{R}^n} \psi_\varepsilon(2(\bk - \bk^{\prime}), t)  (\me^{-2\mi \bk^{\prime} \cdot \bz(\bx)} \varphi(\bx, \bk^{\prime}, t)) \D \bk^{\prime} \Vert_{L^2_{\bk}} \\
&\quad + 2^{n+1}\Vert \int_{\mathbb{R}^n} \psi_\infty(2(\bk - \bk^{\prime}), t) (\me^{-2\mi \bk^{\prime} \cdot \bz(\bx)} \varphi(\bx, \bk^{\prime}, t)) \D \bk^{\prime} \Vert_{L^2_{\bk}}.
\end{split}
\end{equation}
The first term is bounded from $L^2(\mathbb{R}^{2n})$ to itself as $\Psi$ is locally integrable, say, 
\begin{equation}
\Vert \int_{\mathbb{R}^n} \psi_\varepsilon(2(\bk - \bk^{\prime}), t)  (\me^{-2\mi \bk^{\prime} \cdot \bz(\bx)} \varphi(\bx, \bk^{\prime}, t)) \D \bk^{\prime} \Vert_{L^2_{\bk}}  \le \Vert \Psi \cdot \chi_{\varepsilon, R} \Vert_{L^1_{\bk}} \cdot \Vert \varphi(t) \Vert_{L^2_{\bk}},
\end{equation}
and the bound is independent of $\varepsilon$. The second term is also bounded from $L^2(\mathbb{R}^{2n})$ to $L^p(\mathbb{R}^{2n})$, with $1/p = 1/2 + \alpha/n$, owing to {the Hardy-Littlewood-Sobelev theorem}, 
\begin{equation}
\begin{split}
&\Vert \int_{\mathbb{R}^n} \psi_\infty(2(\bk - \bk^{\prime}), t) (\me^{-2\mi \bk^{\prime} \cdot \bz(\bx)} \varphi(\bx, \bk^{\prime}, t)) \D \bk^{\prime}  \Vert_{L^2_{\bk}} \\
&\le \Vert \int_{\mathbb{R}^n} \frac{\chi_{2R, +\infty}(2|\bk - \bk^{\prime}|) \cdot |\varphi(\bx, \bk^{\prime}, t)|}{2^{n-\alpha}|\bk - \bk^{\prime}|^{n-\alpha}} \D \bk^{\prime} \Vert_{L^2_{\bk}} \le C_{p} \Vert \varphi(t) \Vert_{L^p_{\bk}} \le \tilde{C}_p \Vert \varphi(t) \Vert_{L^2_{\bk}}.
\end{split}
\end{equation}
Now let $\varepsilon \to 0$ in Eq.~\eqref{def.pdo_L2_bound}. By combining Eq.~\eqref{eq.L2_embedding}, we obtain that there exists a uniform $K_V$ such that
\begin{equation}
\Vert \Theta_V[\varphi](t) \Vert_2 \le K_V \Vert \varphi(t) \Vert_2.
\end{equation}


A remarkable feature of the oscillatory integral operator is the decay property as the integrand becomes more and more oscillating. As stated by H\"{o}rmander's theorem \cite{bk:Sogge2014-AM188,bk:Stein2016-PMS43}, when $\psi$ is sufficiently smooth and compactly supported,  it has a sharp estimate for a sufficiently large $|\bz(\bx)|$
\begin{equation}\label{eq.Hormander_estimate}
\Vert \Theta_V[\varphi](\bx, \bk, t) \Vert_{L^2_{\bk}} \le C |\bz(\bx)|^{-n/2} \Vert \varphi(t) \Vert_{L^2_{\bk}}.                                                                                                                                                                                                                                                                                                                                                                                                                                                                                                                                                                                                                                
\end{equation}
The physical meaning of Eq.~\eqref{eq.Hormander_estimate} is also clear. When we consider the two-body interacting potential like $V(\bx) = V(|\bx - \bx_A|)$, $\bz(\bx) = \bx - \bx_A$ turns out to be the spatial displacement between two bodies, so that the estimate \eqref{eq.Hormander_estimate} characterizes the decay rate of quantum interaction as the distance $|\bz(\bx)|$ increases. A similar result like Eq.~\eqref{eq.Hormander_estimate} can also be found in our framework, and the decay property will be fully utilized by the stationary phase method as presented in Theorem \ref{thm.stationary_phase_approximation}, which is definitely ignored by HJD \eqref{eq.PDO_splitting}.

\begin{proof}[Proof of Theorem \ref{thm.stationary_phase_approximation}]

It starts by splitting the wavevector $\bk$ into its modulus and orientation parts $\bk^{\prime} = r\sigma$ with the modulus $r=|\bk^\prime|>0$ and the orientation $\sigma=(\sigma_1,\sigma_2,\ldots,\sigma_n) \in \mathbb{S}^{n-1}$, {where $\mathbb{S}^{n-1}$ denotes the $(n-1)$-dimensional spherical surface}, and focusing on the high-frequency component 
\begin{equation}\label{eq:high}
\Theta_V[\varphi] - \Lambda^{<\lambda_0}[\varphi] = \int_{\frac{\lambda_0}{|\bz(\bx)|}}^{+\infty} r^{n-1}\D r \int_{\mathbb{S}^{n-1}}\D \sigma ~\me^{\mi r |\bz(\bx)| \bz^{\prime} \cdot \sigma} \psi(r\sigma, t) \Delta_{r\sigma}[\varphi](\bx, \bk, t ),
\end{equation}
where $\sigma = (\sigma_{1}, \cdots, \sigma_{n})$  represents the orientation of $\bz(\bx)$, $\bz^{\prime} = \bz/|\bz|$ and $\D \sigma$ denotes the  induced Lebesgue measure on $\mathbb{S}^{n-1}$.
After choosing the equatorial plane normal to $\bz^{\prime}$,
the unit sphere $\mathbb{S}^{n-1}$ can be decomposed into 
an upper hemisphere $\mathbb{S}_+^{n-1}$ and a lower one $\mathbb{S}_-^{n-1}$ satisfying $\pm \bz^{\prime} \in \mathbb{S}_\pm^{n-1}$. Accordingly, the inner surface integral of the first kind over $\mathbb{S}^{n-1}$ in Eq.~\eqref{eq:high} equals to the sum of those over $\mathbb{S}_+^{n-1}$ and $\mathbb{S}_-^{n-1}$. 
Without loss of generality, it suffices to assume that $\bz^{\prime} = (0, \dots, 0, 1)$, which be realized by a rotation otherwise. Let us start from the graph 
\begin{equation}
\mathbb{S}_{\pm}^{n-1} = \{\sigma\in\mathbb{R}^n \big| \sigma_n = \pm\phi(\sigma_1, \dots, \sigma_{n-1}), \sigma_i\in[-1,1], i=1,\ldots,n-1\}, 
\end{equation}
with
\begin{equation}
\phi(\sigma_1, \dots, \sigma_{n-1}) = \sqrt{1- \sigma_1^2 - \dots - \sigma_{n-1}^2}
\end{equation}
and take the surface integral of the first kind over the upper hemisphere as an example.
Now the phase function of the integrand becomes 
\begin{equation}
S(\bx, r, \sigma_1,\ldots,\sigma_{n-1}) = r |\bz(\bx)| \bz^{\prime} \cdot \sigma = r|\bz(\bx)| \phi(\sigma_1, \dots, \sigma_{n-1}).
\end{equation}
For such phase function, it can be easily verified that there is only one critical point $\sigma_+ =  (0, \dots, 0, 1)$ satisfying {$(\nabla_\sigma S)(\sigma_+) = 0$}, 
and the determinant of its Hessian matrix at $\sigma_+$ turns out to be 
\begin{equation}
\textup{det} (\textup{Hess}(S)(\sigma_+)) = \mathop{\textup{det}}_{1\le j, k \le n-1} \left(\frac{\partial^2 S}{\partial \sigma_j \partial \sigma_k}(\sigma_+)\right)   = (-r|\bz(\bx)|)^{n-1}.
\end{equation}
In consequence, applying the stationary phase method \cite{bk:Stein2016-PMS43} leads directly to 
\begin{equation*}
\begin{split}
& \int_{\mathbb{S}^{n-1}_+} \me^{\mi S(\bx, r, \sigma)} \psi(r\sigma, t) \Delta_{r \sigma}[\varphi](\bx, \bk, t )\D \sigma \\
& = \me^{\mi S(\bx, r, \sigma_+)}\left(\frac{(2\pi\mi)^{n-1}}{\textup{det}(\textup{Hess}(S)(\sigma_+))}\right)^{\frac{1}{2}}  \psi(r\sigma_+, t) \Delta_{r \sigma_+}[\varphi](\bx, \bk, t ) + R_{\sigma_+}(\bx, \bk, r, t) \\
& = \me^{\mi r |\bz(\bx)|}\left(\frac{2\pi}{\mi r |\bz(\bx)|}\right)^{\frac{n-1}{2}}  \psi(r\sigma_+, t) \Delta_{r\sigma_+}[\varphi](\bx, \bk, t) + R_{\sigma_+}(\bx, \bk, r, t),
\end{split}
\end{equation*}
the first term of which exactly recovers the integrand of $\Lambda^{> \lambda_0}_+$ in Eq.~\eqref{def.stationary_phase_approximation}. That is, the asymptotic of the oscillatory integral over the upper hemisphere is governed by the contribution from the critical point $\sigma_+$. 

It remains to estimate the integral of remainders $\int_{\lambda_0/|\bz(\bx)|}^{+\infty} R_{\sigma_\pm}(\bx, \bk, r, t) r^{n-1} \D r$. Since $\psi(\bk, t) \in C([0, T], C^\infty_0(\mathbb{R}^n))$ with its support contained in a compact ball $B(2R)$, we can replace $\psi$ by $\psi \cdot \chi_{\varepsilon, R} $, with $\varepsilon < \lambda_0/|\bz(\bx)|$. Now we rewrite $R_{\sigma_+}(\bx, \bk, r, t)$ as
\begin{equation}\label{eq.remainder_term}
\begin{split}
R_{\sigma_\pm}(\bx, \bk, r, t)= &~a_\pm(\bx, r, t)  \chi_{\varepsilon, R}(r) \Delta_{r\sigma_\pm}[\varphi](\bx, \bk, t) + b_\pm(\bx , \bk, r, t) \chi_{\varepsilon, R}(r),
\end{split}
\end{equation}
where
\begin{align*}
&a_\pm(\bx, r, t) =  \int_{\mathbb{S}_\pm^{n-1}} \me^{\mi r |\bz(\bx)| \bz^{\prime} \cdot \sigma} \psi(r\sigma, t) \D \sigma - \me^{\pm \mi r |\bz(\bx)|} \left(\frac{2\pi }{\pm \mi r |\bz(\bx)|}\right)^{\frac{n-1}{2}} \psi(r\sigma_\pm, t),\\
&b_\pm(\bx, \bk, r, t) =  \int_{\mathbb{S}_\pm^{n-1}} \me^{\mi r|z(\bx)| \bz^{\prime} \cdot \sigma} \psi(r\sigma, t)  (\Delta_{r\sigma}[\varphi] - \Delta_{r\sigma_\pm}[\varphi]) \D \sigma.
\end{align*}
According to Theorem 7.7.14 in \cite{bk:Hormander2003}, it has an estimate for $a_\pm$ that
\begin{equation}
|a_\pm(\bx, r, t)| \le {C}{(r|z(\bx))|^{-\frac{n+1}{2}}} \le C \lambda_0^{-\frac{n+1}{2}}.
\end{equation}
Thus for the first term in Eq.~\eqref{eq.remainder_term}, it yields that
\begin{equation}\label{eq.estimate_first_term}
\begin{split}
& \Vert \int_{\lambda_0/|\bz(\bx)|}^{+\infty} r^{n-1} a_\pm(\bx, r, t) \chi_{\varepsilon, R}(r)  \Delta_{r\sigma_\pm}[\varphi](\bx, \bk, t) \D r  \Vert_{L^2_{\bk}}^2 \\
&\le \int_{\varepsilon}^{+\infty}  r^{n-1} \chi_{\varepsilon, R}(r)  |a_\pm(\bx, r, t)|^2 \D r   \int_{\varepsilon}^{+\infty} r^{n-1} \chi_{\varepsilon, R}(r) \Vert \Delta_{r\sigma_\pm}[\varphi](\bx, \bk, t) \Vert_{L^2_{\bk}}^2 \D r\\
& \le \frac{2 C^2 R^{2n}}{n^2}  \cdot  \lambda_0^{-(n+1)}   \Vert \varphi(t) \Vert_{L^2_{\bk}}^2.
\end{split}
\end{equation}

For the second term {in Eq.~\eqref{eq.remainder_term}}, it suffices to consider a sufficiently smooth $\varphi$, so that the localization property of oscillatory integrals allows us to only estimate the integral in the neighborhood $U^\pm$ of the stationary phase points $\sigma_\pm = (0, \dots, 0)$. Due to the Morse lemma \cite{bk:Milnor2016}, there exists a  diffeomorphism from $U^+$ to a small neighborhood of $\by_+ = (0, \dots, 0)$. Indeed, since $\phi(0, \dots, 0) = 1$ and $\nabla \phi(0, \dots, 0) = 0$, we have that 
\begin{equation}
\phi(\sigma_1, \dots, \sigma_{n-1}) -1  = \int_0^1 (1-t) \frac{\D^2 \phi}{\D t^2} (t\sigma_1, \dots, t\sigma_{n-1}) \D t = \sum_{i, j} \sigma_i \sigma_j h_{ij}(\sigma_1, \cdots, \sigma_{n-1}),
\end{equation}
where
\begin{equation}
h_{ij}(\sigma_1, \cdots, \sigma_{n-1}) = \int_0^1 (1-t) \partial^2_{ij} \phi(t\sigma_1, \dots, t\sigma_{n-1}) \D t. 
\end{equation}
It notes that $H= (h_{ij})$ is a symmetric matrix and nonsingular at $(0, \dots, 0)$, and so is in $U^+$ by continuity, then there exists a nonsingular $n\times n$ matrix $B(\sigma_1, \dots, \sigma_{n-1}) = (b_{ij}(\sigma_1, \dots, \sigma_{n-1}))$ such that $H = B^\tau B$. Therefore, we can introduce the $\by$-coordinate
\begin{equation}
\by = (y_1, \cdots, y_{n-1})^\tau = B(\sigma_1, \cdots, \sigma_{n-1}) (\sigma_1, \cdots, \sigma_{n-1})^\tau
\end{equation}
so that the phase function turns out to be a quadratic form
\begin{equation}
S(\bx, r, \sigma) = \tilde{S}(\bx, r, \by) = r|\bm{z(\bx)}| (1 - y_1^2 - \dots - y_{n-1}^2).
\end{equation}
By the implicit function theorem, the inverse  conversion $\sigma_i = \kappa_i(y_1, \dots, y_{n-1}) \in C^\infty(\mathbb{R}^{n-1})$ also exists, which satisfies $\kappa_i(0, \dots, 0) = 0$. Therefore, it further has that
\begin{equation}
\sigma_i = \kappa_i(y_1, \dots, y_{n-1}) = \int_0^1 \frac{\D \kappa_i}{\D t}(ty_1, \dots, ty_{n-1}) \D t = \sum_{j= 1}^{n-1} y_j \kappa_{ij}(y_1, \dots, y_{n-1})
\end{equation}
with a suitable $C^\infty$ function $\kappa_{ij}$ that satisfies $\kappa_{ij}(0, \dots, 0) = \frac{\partial \kappa_i}{\partial y_j}(0, \dots, 0)$. 

Now we use Taylor's expansion,
\begin{equation}
\begin{split}
&\varphi(\bx, \bk \pm \frac{r\sigma}{2}, t) - \varphi(\bx, \bk \pm \frac{r\sigma_+}{2}, t) \\
& = \pm\frac{1}{2} \sum_{i=1}^{n-1} r \sigma_i \left(\frac{\partial \varphi}{\partial k_i} + \frac{\partial \phi}{\partial \sigma_i} \frac{\partial \varphi}{\partial k_n}\right)(\bx, \bk \pm \frac{r {\sigma_+}}{2}, t)  + \mathcal{O}( \sigma^2) \\
& = \pm\frac{1}{2} \sum_{i=1}^{n-1} \sum_{j=1}^{n-1} r y_j \kappa_{ij}(y_1, \dots, y_{n-1})\left(\frac{\partial \varphi}{\partial k_i} + \tilde{\phi_i} \frac{\partial \varphi}{\partial k_n}\right)(\bx, \bk \pm \frac{r {\sigma_+}}{2}, t) + \mathcal{O}(\by^2),
\end{split}
\end{equation}
with $\tilde{\phi}_i(y_1, \dots, y_{n-1}) = \frac{ \partial \phi}{\partial \sigma_i} (\sigma_1, \dots, \sigma_{n-1})$,  
and 
\begin{equation}
\Big |\int_{\mathbb{R}^{n-1}} \me^{- \mi \tilde{S}(\bx, r, \by)}  \by^l \tilde{\psi}(r\by, t) \D \by \Big | \le C (r|\bz(\bx)|)^{-\frac{n-1}{2} - \frac{|l|}{2}},
\end{equation}
for a sufficiently large $r |\bz(\bx)|$ and $\tilde{\psi}(r\by, t) = \psi(r\sigma, t)$ (here we adopt the convection that $\by^{l} = y_1^{l_1} \dots y_{n-1}^{l_{n-1}}$ for $l_1 + \dots + l_{n-1} = l$) \cite{bk:Stein2016-PMS43}, one can conclude that $L^2$-norm (in {$\bk$-variable}) of the oscillatory integral $b_+(\bx, \bk, r, t)$ is majorized by
\begin{equation}\label{eq.estimate_second_term}
\Vert b_+(\bx, \bk, r, t) \Vert_{L^2_{\bk}} \le C (r|\bz(\bx)|)^{-\frac{n}{2}} \Vert \nabla_{\bk}\varphi(t) \Vert_{L^2_{\bk}} \cdot \chi_{\varepsilon, R}(r) + \mathcal{O}\left((r|\bz(\bx)|)^{-\frac{n+1}{2}}\right).
\end{equation} 
Combining Eqs.~\eqref{eq.estimate_first_term} and \eqref{eq.estimate_second_term}, we arrive at 
\begin{equation}
\begin{split}
\big \Vert \int_{\frac{\lambda_0}{|\bz(\bx)|}}^{+\infty}   R_{\sigma_+}(\bx, \bk, r, t) r^{n-1} \D r \big \Vert_{L^2_{\bk}} & \le C\lambda_0^{-\frac{n}{2}} \Vert \varphi(t) \Vert_{H^1_{\bk}}  \int_{\frac{\lambda_0}{|\bz(\bx)|}}^{+\infty} r^{n-1}\chi_{\varepsilon, R}(r) \D r \\
&\lesssim \lambda_0^{-\frac{n}{2}} \Vert \varphi(t) \Vert_{H^1_{\bk}},
\end{split}
\end{equation}
which implies Eq.~\eqref{eq.error_estimate_stationary_phase_method}.
\end{proof}

\section{Variance estimation of WBRW-HJD}
\label{sec:WBRW-HJD}


From this section we initialize our discussion on the probabilistic aspect. The probabilistic interpretation of Eq.~\eqref{eq.backward_Wigner} borrows several ideas from the renewal theory, as the exponential distribution $\mathcal{G}(t^{\prime}- t)$ characterizes the arrival time of the random jump and the Wigner kernel $V_W$ the transition of states.  The main difficulty lies in the possible negative values of the Wigner kernel $V_W$ because it cannot be regarded as a transition kernel directly. Nonetheless, when the HJD is adopted and the split Wigner kernels $V_W^\pm$ can be normalized,  the existing BRW models are naturally introduced
and the corresponding first moments solve the Wigner equation \eqref{eq.backward_Wigner}. Furthermore, the probabilistic interpretation of the inner product  \eqref{eq.averaged} is also readily established through a straightforward extension of the probability space. 

Under Assumption {\textbf{(A1)}}, it suffices to replace the Wigner kernels by the truncated ones $V_{W, R}^\pm$
\begin{equation}\label{def.truncated_wigner_kernel}
V_{W,R}^\pm(\bx, \bk, t) = V_W^\pm(\bx, \bk, t) \chi_{0, R}(|\bk|)
\end{equation}
with $B(2R)$ the {\sl minimal} ball that satisfies $\mathcal{K} \subset B(2R)$. Now the truncated Wigner kernels are integrable in $\mathbb{R}^n$,
\begin{equation}
\xi(\bx, t) = \int_{\mathbb{R}^n} V^+_{W, R}(\bx, \bk, t) \D \bk = \int_{\mathbb{R}^n} V^-_{W, R}(\bx, \bk, t) \D \bk < \infty,
\end{equation}
and the anti-symmetry relation \eqref{eq.anti_symmetry} is applied in the second equality. Indeed, $V^\pm_{W, R}$ play a role of transition kernels for two branches since
\begin{equation}
V_{W}(\bx, \bk, t) \chi_{0, R}(|\bk|) = \gamma_0 \frac{\xi(\bx, t)}{\gamma_0} \left[ \frac{V^+_{W, R}(\bx, \bk, t)}{\xi(\bx, t)} - \frac{V^-_{W, R}(\bx, \bk, t) }{\xi(\bx, t)} \right], \label{eq.Wigner_kernel_split}
\end{equation}
where the auxiliary constant $\gamma_0$, i.e., the intensity of the exponential distribution (see Eq.~\eqref{eq:measure2}), 
is chosen such that  
\begin{equation}\label{gamma0}
\gamma_0 \ge \displaystyle{\max_{t\in[0, T]} {\max_{\bx} \xi(\bx, t)}},
\end{equation}
{which has already been stated in Eq.~\eqref{eq.PDO_plus_bound}}. It deserves to mention that the normalizing function $\xi(\bx, t) $ is monotonically non-decreasing as $R$ increases.


Based on different interpretations of the multiplicative functional $\xi(\bx, t)/\gamma_0$, we propose two kinds of stochastic branching walk models, termed the weighted-particle model \cite{ShaoXiong2019} and the signed-particle model \cite{XiongShao2019}, respectively. The former is to interpret $\xi(\bx, t)/\gamma_0$ as the weight function, while the latter is to treat it as the probability to generate offsprings. The main result has been illustrated in Theorem \ref{thm.variance_estimate} and revealed the {discrepancy in variances}. In fact, choosing a larger $\gamma_0$ leads to a variance reduction in the weighted-particle model, but does not influence that of the signed-particle counterpart. 
It should be noted that the content until Definition \ref{def.stochastic_process} below has been well delineated in  \cite{ShaoXiong2019} and we just brief it here for the sake of descriptive integrality.

In order to identify the objects in a family history, we need a sequence. Beginning with an ancestor, denoted by $\langle 0 \rangle$, and we can denote its $i$-th children by  $\langle i \rangle$. Similarly, we can denote the $j$-th child of $i$-the child by $\langle i j \rangle$, and thus $ \langle  i_1 i_2\cdots i_n \rangle$ means $i_n$-th child of  $i_{n-1}$-th child of $\cdots$ of the $i_2$-child of the $i_1$-th child, with $i_n \in \left\{1, 2, 3\right\}$. The ancestor $\langle 0 \rangle$ is omitted here and hereafter for brevity.

The branching particle system considered involves four basic elements: the life-length $\tau$, the position $\bx$, the wavevector $\bk$ and the particle weight $w$. 
\begin{definition}\label{def.family_history}
A family history $\omega$ stands for a random sequence
\begin{equation}\label{Def:family_history}
\omega=\{(\tau_0, Q_0, w_0); (\tau_{1}, Q_{1}, w_{1});  (\tau_{2}, Q_{2}, {w_{2}}); (\tau_{3} ,Q_{3}, w_3); (\tau_{11}, Q_{11}, w_{11}); \cdots\},
\end{equation}
where $Q_i$ stands for $(\bx_i, \bk_i)$ and the tuple $(\tau_i, Q_{i}, w_i)=(\tau_{i}, \bx_{i}, \bk_{i}, w_i)$ appears in a definite order of enumeration. $\tau_{i}$, $\bx_i$, $\bk_i$, $w_i$ denote the life-length, starting position, wavevector and particle weight of the $i$-th particle, respectively. The exact order of $(\tau_i, Q_{i}, w_i)$ is immaterial but is supposed to be fixed. The collection of all family histories is denoted by $\Omega$.
\end{definition}

\begin{definition}
For $\omega=\{(\tau_0, Q_{0}, w_0); (\tau_1, Q_{1}, w_1); (\tau_2, Q_{2}, w_2); (\tau_3, Q_{3}, w_3); \cdots\}$, the subfamily $\omega_{i}$ is the family history of $\langle i \rangle$ and its descendants, as defined by $\omega_{i}=\{(\tau_i, Q_{i}, w_i); (\tau_{i1}, Q_{i1}, w_{i1}); (\tau_{i2}, Q_{i2}, w_{i2}), (\tau_{i3}, Q_{i3}, w_{i3}); \cdots\}$.  The collection of $\omega_i$ is denoted by $\Omega_i$.
\end{definition}

The particles are frozen when hitting the first exit time $T$. 
\begin{definition}\label{def:frozen}
Suppose the family history $\omega$ starts at time $t$ and define the stopping time, termed the arrival time $t_{i}$ of a branching-and-jump event, recursively as 
\begin{equation}
t_{0} = t, ~~t_{i_1} = t+\tau_0, ~~t_{i_1i_2\cdots i_n} = t_{i_1i_2\cdots i_{n-1}}+\tau_{i_1i_2\cdots i_{n-1}}.
\end{equation}
Then a particle $\langle  i_{1}i_{2}\cdots i_{n}\rangle$ is said to be frozen at $T$ if the following conditions hold
\begin{equation}
t_{i_1 i_2 \cdots i_{n}} < T ~~\textup{and}~~ t_{i_1 i_2 \cdots i_{n}} + \tau_{i_1 i_2 \cdots i_{n}} \ge T.
\end{equation}
In particular, when $t+\tau_0\ge T$, the ancestor particle $\langle 0 \rangle$ is frozen.  The collection of frozen particles starting at $t$ is denoted by $\mathcal{E}_t(\omega)$.
\end{definition}

Hereafter we assume that all particles in the branching particle system will move until reaching the frozen states, and still use $\Omega$ to denote the collection of the family history of all frozen particles. Next we will illustrate the probability laws $\Pi_Q^\mathrm{w}$ and $\Pi_Q^\mathrm{s}$ of a random cloud initially concentrated at $Q = (\bx, \bk)$. In general, the position and wavevector of the ancestor particle $\langle 0 \rangle$ are set to be $Q_0 = Q$. 

Now consider the probability of event $\mathrm{E}$ (starting at time $t$ at state $Q$)
\begin{equation}
\mathrm{E} =  \{ \tau_0 \in \mathcal{T}_0, (\tau_{i_1}, \bk_{i_1}) \in \mathcal{T}_1  \times \mathcal{K}_1, \cdots,  (\tau_{i_1\cdots i_n}, \bk_{i_1 \cdots i_n }) \in \mathcal{T}_n\times \mathcal{K}_n\}
\end{equation}
for any Borel set $\mathcal{T}_i$ on $[0, +\infty)$ and $\mathcal{K}_i$ on $ \mathbb{R}^n$, then the probability laws are given by
\begin{equation}\label{def.probability_law}
\begin{split}
\hspace*{-0.05cm} 
\Pr(\mathrm{E}) = &\int_{\mathcal{T}_0} \textup{d} \mathcal{G}(\tau_0)   \int_{\mathcal{K}_1} \textup{d}\bk_{i_1}  \mathrm{K}_{i_1}^{t_{i_1}, \bx_{i_1}}(\bk_{i_1};\bk) \int_{\mathcal{T}_1} \textup{d} \mathcal{G}(\tau_{i_1}) \int_{\mathcal{K}_2} \textup{d}\bk_{i_1 i_2}   \mathrm{K}_{i_1 i_2}^{t_{i_1 i_2}, \bx_{i_1 i_2}}(\bk_{i_1 i_2};\bk_{i_1})  \\
&  \times  \cdots \times \int_{\mathcal{K}_n}  \textup{d}\bk_{i_1\cdots i_n} \mathrm{K}^{t_{i_1\cdots i_{n}}, \bx_{i_1\cdots i_{n}}}_{i_1 \cdots i_n}(\bk_{i_1 \cdots i_{n}}; \bk_{i_1 \cdots i_{n-1}}) \int_{\mathcal{T}_n} \textup{d} \mathcal{G}(\tau_{i_1\cdots i_n}), 
\end{split}
\end{equation}
where $\bx_{i_1\cdots i_{n}} = \bx_{i_1\cdots i_{n-1}}(\tau_{i_1\cdots i_{n-1}})$
with the transition kernels $\mathrm{K}_{i_1 \cdots i_m}^{t^{\prime}, \bx^{\prime}}(\bk; \bk^{\prime})$ given by 
\begin{equation}\label{def.transition_kernel}
\mathrm{K}_{i_1 \cdots i_m}^{t^{\prime}, \bx^{\prime}}(\bk; \bk^{\prime}) =  \left\{
\begin{split}
&\frac{V_{W, R}^+(\bx^{\prime}, (-1)^{i_m}(\bk - \bk^{\prime}), t^{\prime})}{\xi(\bx^{\prime}, t^{\prime})}, &\quad &i_m = 1,2,\\
& \delta(\bk - \bk^{\prime}), &\quad &i_m = 3.
\end{split}
\right.
\end{equation}
The difference lies in the setting of particle weight $w_{i_1 \cdots i_m}$.
\begin{itemize}

\item[(1)] For the weighted particle model, 
\begin{equation}\label{def.wp_particle_weight}
w_{i_1\cdots i_m} = \left\{
\begin{split}
&\frac{\xi(\bx_{i_1\cdots i_m}, t_{i_1\cdots i_m})}{\gamma_0} \cdot \mone_{\mathcal{K}}(\bk_{i_1\cdots i_m}) , &\quad i_m &= 1,2\\
&1, &\quad i_m &= 3,
\end{split}
\right.
\end{equation}

\item[(2)] For the signed particle model, for $i_m = 1$ and $2$,
\begin{equation}\label{def.sp_particle_weight}
w_{i_1\cdots i_m} = \left\{
\begin{split}
&1, \quad \textup{with} ~\Pr = \frac{\xi(\bx_{i_1\cdots i_m}, t_{i_1\cdots i_m})}{\gamma_0} \cdot \mone_{\mathcal{K}}(\bk_{i_1\cdots i_m}), \\
&0, \quad \textup{otherwise},
\end{split}
\right.
\end{equation}
and $w_{i_1\cdots i_m} = 1$ for $i_m = 3$.

\end{itemize}

The setting of initial particle weight $w_0$ depends on the situation. At this stage, it suffices to take $w_0 = 1$. However, later we will show that the initial particle weights may take values in $\{-1, 1\}$, resulting from the importance sampling according to the initial Wigner function $f_0$. Now we illustrate the construction of the stochastic processes $\mathrm{X}_t^\mathrm{w}$ and $\mathrm{X}_t^\mathrm{s}$.

\begin{definition}
\label{def.stochastic_process}
Suppose $(\bx_i, \bk_i)$ is the starting state of a frozen particle $i$ in a given family history $\omega$, and let $\delta_{(\bx, \bk)}$ be the Dirac measure concentrated at  state $(\bx, \bk)$. Then the weighted-particle WBRW is given by
\begin{equation}\label{def.wpWBRW}
\mathrm{X}_t^\mathrm{w}(\omega) = \langle \varphi_T, \sum_{i \in \mathcal{E}_t(\omega)} \hat{w}_{i} \cdot \delta_{(\bx_{i}(T-t_i), \bk_{i})} \rangle =  \sum_{i \in \mathcal{E}_t(\omega)} \hat{w}_i \cdot \varphi_T(\bx_i(T-t_i), \bk_i),
\end{equation}
the cumulative weight $\hat{w}_i \in [-1, 1]$ for $i = \langle i_1 i_2 \cdots i_n \rangle$ is defined by the product of the particle weights $w_{i_1\cdots i_m}$, 
\begin{equation}
\hat{w}_{i} = \prod_{m=1}^n (-1)^{i_m+1}w_{i_1\cdots i_m} , \quad |w_{i_1\cdots i_m}| \le 1,
\end{equation}
where $w_{i_1\cdots i_m}$ are given by \eqref{def.wp_particle_weight}.

Similarly, the signed-particle WBRW is given by
\begin{equation}\label{def.spWBRW}
\mathrm{X}_t^\mathrm{s}(\omega) = \langle \varphi_T, \sum_{i \in \mathcal{E}_t(\omega)} \hat{s}_{i} \cdot \delta_{(\bx_{i}(T-t_i), \bk_{i})} \rangle =  \sum_{i \in \mathcal{E}_t(\omega)} \hat{s}_i \cdot \varphi_T(\bx_i(T-t_i), \bk_i),
\end{equation}
where cumulative weight $\hat{s}_i \in \{-1, 0, 1\}$ for $i = \langle i_1 i_2 \cdots i_n \rangle$ is defined by
\begin{equation}
\hat{s}_i =  \prod_{m=1}^n (-1)^{i_m+1} w_{i_1\cdots i_m}, \quad w_{i_1\cdots i_m} \in \{-1, 0, 1\},
\end{equation}
where $w_{i_1\cdots i_m}$ are given by \eqref{def.sp_particle_weight}.

\end{definition}



According to Eq.~\eqref{def.probability_law}, it's easy to verify the Markov property of $\Pi_Q^\mathrm{w}$, which also holds for $\Pi_Q^\mathrm{s}$.
\begin{equation}\label{def.Markov_property}
\begin{split}\mathrm
\mathrm{\Pi}^\mathrm{w}_{Q} X^\mathrm{w}_t X^\mathrm{w}_{t+\tau_0} & =\Pi_Q^\mathrm{w} X^\mathrm{w}_t (\Pi^\mathrm{w}_{Q_{i_1}} X^\mathrm{w}_{t+\tau_0}) \\
& = \int_{\Omega} \left(X^\mathrm{w}_t(\omega) \int_{\Omega_{i_1}} X^\mathrm{w}_{t+\tau_0}(\omega_{i_1}) \Pi^\mathrm{w}_{Q_{i_1}} (\D \omega_{i_1})\right)\mathrm{\Pi}^\mathrm{w}_{Q}(\D \omega).
\end{split}
\end{equation}

\begin{definition}
\label{def:mom}
The first moments of $\mathrm{X}_t^\mathrm{w}$ and $\mathrm{X}_t^\mathrm{s}$ are denoted by
\begin{equation}
\phi_\mathrm{w}^{(1)}(\bx, \bk, t) =\Pi_Q^\mathrm{w} \mathrm{X}_t^\mathrm{w}, \quad \phi_\mathrm{s}^{(1)}(\bx, \bk, t) =\Pi_Q^\mathrm{s} \mathrm{X}_t^\mathrm{s}, 
\end{equation}
and the second moments are
\begin{equation}
\phi_\mathrm{w}^{(2)}(\bx, \bk, t) =\Pi_Q^\mathrm{w} (\mathrm{X}_t^\mathrm{w})^2, \quad \phi_\mathrm{s}^{(2)}(\bx, \bk, t) =\Pi_Q^\mathrm{s} (\mathrm{X}_t^\mathrm{s})^2, 
\end{equation}
In addition, the variances are defined as
\begin{align}\label{def.Xt_second_moment}
& \Delta \phi^{(2)}_\mathrm{w}(\bx, \bk, t)= \Pi_Q^\mathrm{w} (\mathrm{X}_t^\mathrm{w} - \phi_\mathrm{w}^{(1)}(\bx, \bk, t))^2,\\
& \Delta \phi^{(2)}_\mathrm{s}(\bx, \bk, t) = \Pi_Q^\mathrm{s} ((\mathrm{X}_t^\mathrm{s} - \phi_\mathrm{s}^{(1)}(\bx, \bk, t))^2.
\end{align}
\end{definition}
Before proceeding to the proof of Theorem \ref{thm.variance_estimate}, we require the following two lemmas.
\begin{lemma}[Backward Gr\"{o}nwall's inequality]\label{lemma.Gronwall_ineq} 
Suppose $\beta > 0$ and $u$ satisfies the integral inequality
\begin{equation}
u(t) \le \alpha(t) + (1+\frac{\beta}{\gamma_0} ) \int_t^T \D \mathcal{G}(t^{\prime} - t) u(t^{\prime}),
\end{equation}
then
\begin{equation}\label{eq.Gronwall_ineq}
u(t) \le \alpha(t) + (\gamma_0 + \beta) \int_t^T \me^{\beta( t^{\prime} - t)}\alpha(t^{\prime}) \D t^{\prime}
\end{equation}
\end{lemma}
\begin{proof}
Let 
\begin{equation}\label{eq.conversion}
\tilde{u}(t) = \me^{-\gamma_0 t} u(t), \quad \tilde{\alpha}(t) = \me^{-\gamma_0 t} \alpha(t),
\end{equation}
it yields
\begin{equation}
\tilde{u}(t) \le \tilde{\alpha}(t) + (\gamma_0 + \beta) \int_t^T \tilde{u}(t^{\prime}) \D t^{\prime}.
\end{equation}

By the Gr\"{o}nwall's inequality, we have
\begin{equation}\label{eq.Gronwall_ineq_2}
\tilde{u}(t) \le \tilde{\alpha}(t) + (\gamma_0 + \beta) \int_t^T \me^{\gamma_0(t^{\prime}-t)+ \beta( t^{\prime} - t) }\tilde{\alpha}(t^{\prime}) \D t^{\prime}.
\end{equation}
Substituting Eq.~\eqref{eq.conversion} into Eq.~\eqref{eq.Gronwall_ineq_2} yields Eq.~\eqref{eq.Gronwall_ineq}.
\end{proof}

\begin{lemma}[Prior $L^2$-estimate]\label{lemma.L2_estimate} 
Suppose $\varphi_T \in L^2(\mathbb{R}^{n} \times \mathbb{R}^{n})$ and the pseudo-differential operator $\Theta_V$ is bounded from $L^2(\mathbb{R}^{n} \times \mathbb{R}^{n})$ to itself, say, $\Vert \Theta_V[\varphi](t) \Vert_2 \le K_V \Vert \varphi(t) \Vert_2 $, then for a given $T < \infty$, 
\begin{equation}\label{eq.L2_estimate}
\Vert \varphi (t) \Vert_2 \le \me^{K_V(T-t)} \Vert \varphi_T \Vert_2, \quad t \in [0, T].
\end{equation}
\end{lemma}

\begin{proof}
The operator semigroup $T(t) = \me^{t(\mathcal{A}+\gamma_0)}$ is an isometry from $L^{2}(\mathbb{R}^n)$ to itself. Thus by the triangular inequality and the extended Minkowski's inequality, it has that
\begin{equation}\label{eq.prior_L2_estimate}
\begin{split}
\Vert \varphi(t) \Vert_2 & \le \me^{-\gamma_0(T-t)} \Vert \varphi_T \Vert_2 + \Vert \int_t^T \D \fG(t^{\prime}-t) \{-\frac{\Theta_V[\varphi](t^{\prime})}{\gamma_0}+ \varphi(t^{\prime}) \} \Vert_2\\
& \le \me^{-\gamma_0(T-t)} \Vert \varphi_T \Vert_2 + \int_t^T \D \fG(t^{\prime}-t) \{\Vert -\frac{\Theta_V[\varphi](t^{\prime})}{\gamma_0} + \varphi(t^{\prime}) \Vert_2\} \\
&\le \me^{-\gamma_0(T-t)} \Vert \varphi_T \Vert_2  + (1+\frac{K_V}{\gamma_0})\int_t^T \D \fG(t^{\prime}-t)~ \Vert \varphi(t^{\prime}) \Vert_2.
\end{split}
\end{equation}
Thus by the Lemma \ref{lemma.Gronwall_ineq}, we arrive at 
\begin{equation}
\frac{\Vert \varphi(t) \Vert_2}{\Vert \varphi_T \Vert_2} \le \me^{-\gamma_0(T-t)} + (\gamma_0+K_V)\int_t^T \me^{K_V(t^{\prime}-t) -\gamma_0(T-t^{\prime})}\D t^{\prime} = \me^{K_V (T-t)}.
\end{equation}
\end{proof}

\begin{proof}[Proof of the first part of Theorem \ref{thm.variance_estimate}]

We first consider the weighted-particle part. To estimate $\Delta \phi_\mathrm{w}^{(2)}(\bx, \bk, t)$, it starts from the fact that 
\begin{equation}
\mathrm{\Pi}^\mathrm{w}_{Q} (\mathrm{X}_t^\mathrm{w})^2  = \Pi_Q^\mathrm{w} \left( \mone_{\mathrm{E}_t} (\mathrm{X}_t^\mathrm{w})^2 \right) +\Pi_Q^\mathrm{w} \left( \mone_{\mathrm{E}_t^c} (\mathrm{X}_t^\mathrm{w})^2  \right),
\end{equation}
where the events are $\mathrm{E}_t=\left\{\tau_0: t+\tau_0\ge T\right\}\cap\Omega$ and  $\mathrm{E}_t^c=\left\{\tau_0: t+\tau_0 < T\right\}\cap\Omega$. The second term is expanded as 
\begin{equation*}
\begin{split}
&\mathrm{\Pi}^\mathrm{w}_{Q} \left( \mone_{\mathrm{E}_t^c} (\mathrm{X}_t^\mathrm{w})^2  \right) = \sum_{i=1}^3  \int_{ \mathrm{E}_t^c} \left( w^2_i\int_{\Omega_i}  (X_{t+\tau_0}^\mathrm{w})^2(\omega_i) \Pi^\mathrm{w}_{Q_i}(\D \omega_i) \right)\Pi_Q^\mathrm{w}(\D \omega) \\
&+ \sum_{i \ne j}  \int_{\mathrm{E}_t^c}  \left( (-1)^{i+j}w_i w_j  \int_{\Omega_i} X^\mathrm{w}_{t+\tau_0}(\omega_i) \Pi^\mathrm{w}_{Q_i}(\D \omega_i) \int_{\Omega_j} X^\mathrm{w}_{t+\tau_0} (\omega_j) \Pi^\mathrm{w}_{Q_j}(\D \omega_j) \right)\mathrm{\Pi}^\mathrm{w}_{Q}(\D \omega). 
\end{split}
\end{equation*}
Thus the second moment {$\Pi^\mathrm{w}_Q (\mathrm{X}_t^\mathrm{w})^2$} satisfies the following renewal-type equation.
\begin{equation}\label{def.wpWBRW_renewal_eq}
\begin{split}
\mathrm{\Pi}^\mathrm{w}_{Q} (\mathrm{X}_t^\mathrm{w})^2  = & \me^{- \gamma_0(T-t)} \varphi^2_T(\bx(T-t), \bk) + \int_{0}^{T-t}  \D \fG(\tau_0) B_\mathrm{w}[\phi_\mathrm{w}^{(2)}](\bx(\tau_0), \bk, t+\tau_0)\\
&+ \int_{0}^{T-t}  \D \fG(\tau_0)  C(\bx(\tau_0), \bk, t+\tau_0), 
\end{split}
\end{equation}
where the operator $B_\mathrm{w}$ (the diagonal component) is given by
\begin{equation*}
\begin{split}
B_\mathrm{w}[\phi_\mathrm{w}^{(2)}](\bx, \bk, t) =  \frac{\xi(\bx, t)}{\gamma^2_0} \Theta^-_V[\phi_\mathrm{w}^{(2)}](\bx, \bk, t) + \frac{\xi(\bx, t)}{\gamma^2_0} \Theta^+_V[\phi_\mathrm{w}^{(2)}](\bx, \bk , t) + \phi_\mathrm{w}^{(2)}(\bx, \bk, t),
\end{split}
\end{equation*}
and the correlated term $C(\bx, \bk, t)$ reads 
\begin{equation*}
C(\bx, \bk, t) = -\frac{2}{\gamma_0}  \Theta_V[\varphi](\bx, \bk, t) \cdot \varphi(\bx, \bk, t) -\frac{2}{\gamma_0^2} \Theta_V^-[\varphi](\bx, \bk, t) \cdot  \Theta_V^+[\varphi](\bx, \bk, t).
\end{equation*}

By the triangular inequality and Young's inequality, it's readily to verify that $B_\mathrm{w}$ is bounded operator from $L^1(\mathbb{R}^{n}) \times L^1_0(\mathbb{R}^{n})$ to itself, 
\begin{equation}\label{eq.B_w_bound}
\begin{split}
 \Vert B_\mathrm{w}[\phi_\mathrm{w}^{(2)}](t) \Vert_1 &\le \frac{\breve{\xi}}{\gamma^2_0} \Vert \Theta_V^-[\phi_\mathrm{w}^{(2)}](t) \Vert_1 + \frac{\breve{\xi}}{\gamma^2_0} \Vert \Theta_V^+[\phi_\mathrm{w}^{(2)}](t) \Vert_1 + \Vert \phi_\mathrm{w}^{(2)}(t) \Vert_1 \\
 & \le (1+\frac{2\breve{\xi}^2}{\gamma_0^2}) \Vert \phi_\mathrm{w}^{(2)}(t) \Vert_1.
\end{split}
\end{equation}
Also, by Cauchy-Schwarz inequality, it yields
\begin{equation}\label{eq.C_bound}
\begin{split}
\Vert C(t) \Vert_1 &\le \frac{2}{\gamma_0} \Vert \Theta_V[\varphi](t) \Vert_2 \cdot \Vert \varphi(t) \Vert_2 + \frac{2}{\gamma_0^2}\Vert \Theta^-_V[\varphi](t) \Vert_2 \cdot \Vert \Theta^+_V[\varphi](t) \Vert_2\\
&\le \frac{2K_V}{\gamma_0} \Vert \varphi(t) \Vert_2^2 + \frac{2\breve{\xi^2}}{\gamma_0^2} \Vert \varphi(t) \Vert_2^2.
\end{split}
\end{equation}

Next we turn to analyze the $L^1$-boundness of $\Delta \phi^{(2)}_\mathrm{w}(\bx, \bk, t)$, which satisfies the following renewal-type equation according to Eq.~\eqref{def.wpWBRW_renewal_eq}, 
\begin{equation}\label{eq.wpBRW_variation} 
\begin{split}
\Delta \phi^{(2)}_\mathrm{w}(\bx, \bk, t)  =& \me^{-\gamma_0(T-t)} \varphi^2_T(\bx(T-t),  \bk)  -\varphi^2(\bx, \bk, t) \\
&+ \int_{0}^{T-t} \D \mathcal{G}(\tau)( B_\mathrm{w}[\varphi^2](\bx(\tau), \bk, t+\tau) + C(\bx(\tau), \bk, t+\tau))\\
&+ \int_{0}^{T-t} \D \mathcal{G}(\tau) B_\mathrm{w}[\Delta \phi_\mathrm{w}^{(2)}](\bx(\tau), \bk, t+\tau).
\end{split}
\end{equation}

By integrating Eq.~\eqref{eq.wpBRW_variation} in $\mathbb{R}^n \times \mathbb{R}^n$ and using the triangular inequality, the extended Minkowski's inequality and Eq.~\eqref{eq.B_w_bound}, it has 
\begin{equation}\label{eq.variation_inequality}
\begin{split}
\Vert \Delta \phi_\mathrm{w}^{(2)}(t) \Vert_1 \le& ~\me^{-\gamma_0(T-t)}  \Vert \varphi_T \Vert^2_2 - \Vert \varphi(t) \Vert_2^2 +\int_t^T \D \mathcal{G}(t^{\prime}-t) \Vert C(t^{\prime}) \Vert_1\\
& +\int_t^T \D \mathcal{G}(t^{\prime}-t) ( \Vert  B_\mathrm{w}[\varphi^2](t^{\prime}) \Vert_1 + \Vert B_\mathrm{w}[\Delta \phi_\mathrm{w}^{(2)}](t^{\prime}) \Vert_1 ) \\
 \le&~\me^{-\gamma_0(T-t)}  \Vert \varphi_T \Vert^2_2 - \Vert \varphi(t) \Vert_2^2 +\int_t^T \D \mathcal{G}(t^{\prime}-t) \Vert C(t^{\prime}) \Vert_1\\
 &+ (1+\frac{2 \breve{\xi}^2}{\gamma_0^2}) \int_t^T \D \mathcal{G}(t^{\prime}-t) ( \Vert \Delta \phi_\mathrm{w}^{(2)}(t^{\prime}) \Vert_1 + \Vert \varphi(t^{\prime}) \Vert_2 ).
\end{split} 
\end{equation}

In addition, according to Eqs.~\eqref{eq.B_w_bound} and \eqref{eq.C_bound} and Lemma \ref{lemma.L2_estimate}, it yields
\begin{equation}\label{eq.C_L1_bound}
\begin{split}
\int_t^{T}\D \mathcal{G}(t^{\prime}-t) \Vert C(t^{\prime}) \Vert_1 &\le \frac{2K_V\gamma_0 +2\breve{\xi}^2}{\gamma_0} \int_t^T \me^{2K_V(T-t^{\prime})}  \me^{-\gamma_0(t^{\prime}-t)} \D t^{\prime} \cdot \Vert \varphi_T \Vert^2_2\\
& = \frac{(2K_V\gamma_0 +2\breve{\xi}^2)(\me^{2K_V(T-t)} - \me^{-\gamma_0(T-t)})}{(2K_V+\gamma_0)\gamma_0}   \Vert \varphi_T \Vert_2^2.
\end{split}
\end{equation}

Let 
\begin{equation}
{u(t) = \frac{\Vert \Delta \phi_\mathrm{w}^{(2)}(t) \Vert_1 + \Vert \varphi(t) \Vert_2^2}{\Vert \varphi_T \Vert_2^2}},  \quad \alpha_0 = \frac{2K_V\gamma_0 + 2\breve{\xi}^2}{(2K_V+\gamma_0)\gamma_0},
\end{equation}
then Eq.~\eqref{eq.variation_inequality} is cast into
\begin{equation}
u(t) \le \alpha_0 \me^{2K_V(T-t)} + (1-\alpha_0) \me^{-\gamma_0(T-t)} + (1+\frac{2\breve{\xi}^2}{\gamma_0^2})\int_t^T \D \mathcal{G}(t^{\prime}-t) u(t^{\prime}).
\end{equation}
By using Lemma \ref{lemma.Gronwall_ineq}, it yields
\begin{equation}\label{key_estimate_ut}
u(t) \le \frac{K_V \gamma_0 + \breve{\xi}^2}{K_V \gamma_0 - \breve{\xi}^2} \me^{2K_V(T-t)} - \frac{2\breve{\xi}^2}{K_V\gamma_0 -\breve{\xi}^2} \me^{\frac{2\breve{\xi}^2}{\gamma_0}(T-t)}.
\end{equation}
Here we use the following relations
\begin{align*}
&(\gamma_0+\frac{2\breve{\xi}^2}{\gamma_0})\int_t^T \me^{2K_V(T-t^{\prime})} \me^{\frac{2\breve{\xi}^2}{\gamma_0}(t^{\prime}-t)}\D t^{\prime} = \frac{\gamma_0^2+2\breve{\xi}^2}{2K_V\gamma_0 - 2\breve{\xi}^2} (\me^{2K_V(T-t)} - \me^{\frac{2\breve{\xi}^2}{\gamma_0}(T-t)}),\\
&(\gamma_0+\frac{2\breve{\xi}^2}{\gamma_0})\int_t^T \me^{-\gamma_0(T-t^{\prime})} \me^{\frac{2\breve{\xi}^2}{\gamma_0}(t^{\prime}-t)}\D t^{\prime} = \me^{\frac{2\breve{\xi}^2}{\gamma_0}(T-t)} - \me^{-\gamma_0(T-t)}.
\end{align*}
Consequently, the $L^1$-boundedness of $\Vert \Delta \phi_\mathrm{w}^{(2)}(t) \Vert_1$ for $t \in [0, T]$ is obtained.
\begin{equation*}\label{eq.wp_variation_L1_bound}
\Vert \Delta \phi_\mathrm{w}^{(2)}(t) \Vert_1 \le \left(\frac{K_V \gamma_0 + \breve{\xi}^2}{K_V \gamma_0 - \breve{\xi}^2} \me^{2K_V(T-t)} - \frac{2\breve{\xi}^2}{K_V\gamma_0 -\breve{\xi}^2} \me^{\frac{2\breve{\xi}^2}{\gamma_0}(T-t)}\right) \Vert \varphi_T\Vert_2^2 - \Vert \varphi(t) \Vert_2^2.
\end{equation*}

Finally, when $K_V\gamma_0  > \breve{\xi}^2$, it has
\begin{equation*}\label{eq.wp_variation_L1_bound_2}
\begin{split}
\frac{\Vert \Delta \phi_\mathrm{w}^{(2)}(t) \Vert_1}{\Vert \varphi_T \Vert_2^2} &\le \me^{2K_V(T-t)} - \frac{\Vert \varphi(t) \Vert_2^2}{\Vert \varphi_T \Vert_2^2} + \frac{2\breve{\xi}^2}{K_V\gamma_0 - \breve{\xi}^2}\me^{2K_V(T-t)}(1-\me^{-2(K_V-\frac{\breve{\xi}^2}{\gamma_0})(T-t)}) \\
& \le \me^{2K_V(T-t)} - \frac{\Vert \varphi(t) \Vert_2^2}{{\Vert \varphi_T \Vert_2^2}} + \frac{4\breve{\xi}^2(T-t)}{\gamma_0} \me^{2K_V(T-t)},
\end{split}
\end{equation*}
so that the variance is governed by the leading term $\me^{2K_V(T-t)}$. In contrast, when $K_V\gamma_0  < \breve{\xi}^2$,
\begin{equation*}
\begin{split}
\frac{\Vert \Delta \phi_\mathrm{w}^{(2)}(t) \Vert_1}{\Vert \varphi_T \Vert_2^2} &\le \me^{\frac{2\breve{\xi}^2}{\gamma_0}(T-t)} - \frac{\Vert \varphi(t) \Vert_2^2}{\Vert \varphi_T \Vert_2^2} + \frac{\breve{\xi}^2+K_V\gamma_0}{\breve{\xi}^2 - K_V\gamma_0 }\me^{\frac{2\breve{\xi}^2}{\gamma_0}(T-t)}(1-\me^{-2(\frac{\breve{\xi}^2}{\gamma_0}-K_V)(T-t)}) \\
& \le \me^{\frac{2\breve{\xi}^2}{\gamma_0}(T-t)} - \frac{\Vert \varphi(t) \Vert_2^2}{{\Vert \varphi_T \Vert_2^2}} + \frac{2(\breve{\xi}^2+K_V \gamma_0)(T-t)}{\gamma_0} \me^{\frac{2\breve{\xi}^2}{\gamma_0}(T-t)},
\end{split}
\end{equation*}
so that the variance is governed by the leading term $\me^{\frac{2\breve{\xi}^2}{\gamma_0}(T-t)}$ instead. In this way, we have completed the proof for the weighted-particle model.
\end{proof}


The proof for the signed-particle situation is quite similar and we only need to outline the differences.

\begin{proof}[Proof of the second part of Theorem \ref{thm.variance_estimate}]

The first step is to derive the renewal-type equation for the second moment. Since
\begin{equation*}
\begin{split}
&\mathrm{\Pi}^\mathrm{s}_{Q} \mone_{\mathrm{E}_t^c} (\mathrm{X}_t^\mathrm{s})^2 = \sum_{i=1}^3  \int_{\mathrm{E}_t^c}  \left( w^2_i\int_{\Omega_i}  (\mathrm{X}_{t+\tau_0}^s)^2(\omega_i) \Pi^\mathrm{s}_{Q_i}(\D \omega_i) \right)\Pi_Q^\mathrm{s}(\D \omega)\\
&+ \sum_{i \ne j}  \int_{\mathrm{E}_t^c} \left((-1)^{i+j}w_i w_j  \int_{\Omega_i} \mathrm{X}^s_{t+\tau_0}(\omega_i) \Pi^\mathrm{s}_{Q_i}(\D \omega_i) \int_{\Omega_j} \mathrm{X}^s_{t+\tau_0} (\omega_j) \Pi^\mathrm{s}_{Q_j}(\D \omega_j)\right)\mathrm{\Pi}^\mathrm{s}_{Q}(\D \omega)
\end{split}
\end{equation*}
it's readily obtained that 
\begin{equation}
\begin{split}
\mathrm{\Pi}^\mathrm{s}_{Q} (\mathrm{X}_t^\mathrm{s})^2  = &  \me^{- \gamma_0(T-t)} \varphi^2_T(\bx(T-t), k) + \int_{0}^{T-t}  \D \fG(\tau_0) B_\mathrm{s}[\phi_\mathrm{s}^{(2)}](\bx(\tau_0), \bk, t+\tau_0)\\
&+ \int_{0}^{T-t}  \D \fG(\tau_0)  C(\bx(\tau_0), \bk, t+\tau_0), 
\end{split}
\end{equation}
Accordingly,  the operator $B_s$ is given by
\begin{equation}\label{eq.L1_bound_Bs}
B_\mathrm{s}[\phi_\mathrm{s}^{(2)}](\bx, \bk, t) = \frac{1}{\gamma_0}\Theta^-_V[\phi_\mathrm{s}^{(2)}](\bx, \bk , t) + \frac{1}{\gamma_0}\Theta^+_V[\phi_\mathrm{s}^{(2)}](\bx, \bk , t) + \phi_\mathrm{s}^{(2)}(\bx, \bk, t), 
\end{equation}
which is also a bounded operator from $L^1(\mathbb{R}^{n}\times \mathbb{R}^n)$ to itself,
\begin{equation}\label{eq.bound_Bs}
\Vert B_s [\phi_\mathrm{s}^{(2)}](t) \Vert_1 \le (1+\frac{2\breve{\xi}}{\gamma_0}) \Vert \phi_\mathrm{s}^{(2)}(t) \Vert_1.
\end{equation}

Second, the renewal-type equation for the variance $\Delta \phi_\mathrm{s}^{(2)}(\bx, \bk, t)$
\begin{equation}
\begin{split}
\Delta \phi^{(2)}_s(\bx, \bk, t)  =& ~\me^{-\gamma_0(T-t)} \varphi^2_T(\bx(T-t), k)  -\varphi^2(\bx, \bk, t) \\
&+ \int_{0}^{T-t} \D \mathcal{G}(\tau) ( B_\mathrm{s}[\varphi^2](\bx(\tau), \bk, t+\tau) + C(\bx(\tau), \bk, t+\tau) ) \\
&+ \int_{0}^{T-t} \D \mathcal{G}(\tau) B_\mathrm{s}[\Delta \phi_\mathrm{s}^{(2)}](\bx(\tau), \bk, t+\tau),
\end{split}
\end{equation}
then by the extended Minkowski's inequality and Eq.~\eqref{eq.L1_bound_Bs}, it yields.
\begin{equation}\label{eq.sp_variation_inequality}
\begin{split}
\Vert \Delta \phi_\mathrm{s}^{(2)}(t) \Vert_1 \le &~\me^{-\gamma_0(T-t)}  \Vert \varphi_T \Vert^2_2 -  \Vert \varphi(t) \Vert_2^2+\int_t^T \D \mathcal{G}(t^{\prime}-t) \Vert C(t^{\prime}) \Vert_1\\
 &+ (1+\frac{2 \breve{\xi}}{\gamma_0}) \int_t^T \D \mathcal{G}(t^{\prime}-t) ( \Vert \Delta \phi_\mathrm{s}^{(2)}(t^{\prime}) \Vert_1 + \Vert \varphi(t^{\prime}) \Vert_2 ).
\end{split} 
\end{equation}
By the Lemma \ref{lemma.Gronwall_ineq} and Eqs.~\eqref{eq.C_bound} and \eqref{eq.bound_Bs}, we obtain
\begin{equation*}
\begin{split}
\frac{\Vert \Delta \phi_\mathrm{s}^{(2)}(t) \Vert_1}{\Vert \varphi_T \Vert^2} +\frac{\Vert \varphi(t) \Vert_2^2}{\Vert \varphi_T \Vert_2^2} &\le \frac{\breve{\xi}\gamma_0 +\breve{\xi}^2}{\breve{\xi}\gamma_0 - K_V \gamma_0}\me^{2\breve{\xi}(T-t)} - \frac{\breve{\xi}^2+K_V \gamma_0}{\breve{\xi}\gamma_0 - K_V \gamma_0} \me^{2K_V(T-t)} \\
&\le \me^{2\breve{\xi}(T-t)} + \frac{\breve{\xi}^2 +K_V\gamma_0}{\breve{\xi}\gamma_0 - K_V\gamma_0} \me^{2\breve{\xi}(T-t)}(1 -\me^{-2(\breve{\xi} -K_V)(T-t)}) \\
&\le (1+2(K_V+\frac{\breve{\xi}^2}{\gamma_0})(T-t) )  \me^{2\breve{\xi}(T-t)},
\end{split}
\end{equation*}
{which completes the proof}.
\end{proof}

We are interested in the asymptotic behavior of the variances when $\gamma_0 \to \infty$. For the weighted-particle model, we have
\begin{equation} 
\Vert \Pi_Q^\mathrm{w} (\mathrm{X}_t^\mathrm{w} - \varphi(t))^2 \Vert_1 \lesssim \me^{2K_V (T-t)} \Vert \varphi_T \Vert_2^2 - \Vert \varphi(t) \Vert_2^2,
\end{equation}
so that the variance is uniformly bounded as $\gamma_0$ increases. Whereas for the signed-particle counterpart, 
\begin{equation}
\Vert  \Pi_Q^\mathrm{s} (\mathrm{X}_t^\mathrm{s} - \varphi(t))^2  \Vert_1\lesssim (1+2K_V (T-t))\me^{2\breve{\xi}(T-t)} \Vert \varphi_T \Vert_2^2 - \Vert \varphi(t) \Vert_2^2.
\end{equation}
The leading term is $\me^{2\breve{\xi}(T-t)}$, regardless of the choice of $\gamma_0$.

\begin{remark}
It notes that Eq.~\eqref{eq.PDO_plus_bound} holds directly when $V_W^+(\bx, \bk, t)$ is integrable with respect to $\bk$ for any $\bx \in \mathbb{R}^n$ and $t\in [0, T]$, owing to {Young's inequality}. In this situation, we only require $\varphi \in C([0, T], L^2(\mathbb{R}^n) \times L^2(\mathbb{R}^n))$.
\end{remark}


So far we have given the probabilistic interpretation of the mild solution of the backward Wigner equation by introducing  the stochastic process $(\mathrm{X}_t^\mathrm{w},\Pi_Q^\mathrm{w})$ and $(X_t^\mathrm{s}, \Pi^\mathrm{s}_Q)$ on $\mathscr{B}_{\Omega}$ for a given initial state $(\bx, \bk)$. The probabilistic interpretation of the weak solution of the Wigner equation \eqref{eq.Wigner} can be constructed by an extension of the probability spaces $(\Omega, \mathscr{B}_{\Omega}, \Pi^w_Q)$ and $(\Omega, \mathscr{B}_{\Omega}, \Pi^s_Q)$,
\begin{equation}
\hat{\Omega} = \mathbb{R}^n \times \mathbb{R}^n \times \Omega, ~~ \hat{\mathscr{B}}_{\Omega} = \mathsf{R}^n \otimes \mathsf{R}^n \otimes \mathscr{B}_{\Omega}, ~~ \Pi^\mathrm{w} = \lambda_0 \otimes\Pi_Q^\mathrm{w}, ~~ \Pi^\mathrm{s} = \lambda_0 \otimes\Pi_Q^\mathrm{s}
 \end{equation}
where $\mathsf{R}^n \otimes \mathsf{R}^n \otimes \mathscr{B}_{\Omega}$ is the product Borel extension of $\hat{\Omega}$ and the probability measure $\lambda_0$ is given by 
\begin{equation}
\D \lambda_0 = f_I(\bx, \bk) \D \bx \D \bk, \quad f_I = |f_0|/{\Vert f_0 \Vert_1},
\end{equation}
and
\begin{align}
\Pi^\mathrm{w}  X^\mathrm{w}_t = \lambda_0 \otimes\Pi_Q^\mathrm{w}(X^\mathrm{w}_t) &= \int_{\mathbb{R}^{n} \times \mathbb{R}^{n}}  f_I(\bx, \bk) \left(\int_{\Omega} X^\mathrm{w}_t(\omega)\Pi_Q^\mathrm{w} (\D \omega)\right)\D \bx \D \bk,\\
\Pi^\mathrm{s}  X^\mathrm{s}_t = \lambda_0 \otimes\Pi_Q^\mathrm{s}(X^\mathrm{s}_t) &= \int_{\mathbb{R}^{n} \times \mathbb{R}^{n}} f_I(\bx, \bk) \left(\int_{\Omega} X^\mathrm{s}_t(\omega)\Pi_Q^\mathrm{s} (\D \omega)\right) \D \bx \D \bk.
\end{align}
Thus the inner product $\langle \varphi_T, f_T \rangle = \langle \varphi_0, f_0 \rangle$ can be represented by
\begin{equation}
\langle \varphi_T, f_T \rangle ={\Pi}^\mathrm{w} (s \cdot X^\mathrm{w}_0) = {\Pi}^\mathrm{s} (s \cdot X^\mathrm{s}_0),
\end{equation}
where $s$ is short for the particle sign function
\begin{equation}
s(\bx, \bk) =  f_0(\bx, \bk)/f_I(\bx, \bk).
\end{equation}

According to Theorem \ref{thm.variance_estimate}, it's readily to obtain the variance estimation for the inner product problem.
\begin{definition}
The variances of {$\Pi^\mathrm{w} (s \cdot X)$ and $\Pi^\mathrm{s}(s \cdot X)$} are defined by that
\begin{equation}
\textup{Var}\left({\Pi}^\mathrm{w} (s \cdot \mathrm{X}_t^\mathrm{w})\right) = \lambda_0 \otimes\Pi_Q^\mathrm{w} (s\cdot \mathrm{X}_t^\mathrm{w} - \lambda_0 \otimes\Pi_Q^\mathrm{w}(s\cdot \mathrm{X}_t^\mathrm{w}))^2
\end{equation}
and 
\begin{equation}
\textup{Var}\left({\Pi}^\mathrm{s} (s \cdot \mathrm{X}_t^\mathrm{s})\right) = \lambda_0 \otimes\Pi_Q^\mathrm{s}(s\cdot \mathrm{X}_t^\mathrm{s} - \lambda_0 \otimes\Pi_Q^\mathrm{s}(s\cdot \mathrm{X}_t^\mathrm{s}))^2,
\end{equation}
respectively.
\end{definition}

Now we prove the bounds of the variance for both the weighted-particle model and the signed-particle model.
\begin{theorem}[Variance estimation for the inner product problem]\label{thm.variance_estimation}
Suppose $\Vert f_I \Vert_\infty < \infty$ and there exists a positive constant $M_s>0$ such that $s \le M_s$ holds almost surely in $\mathbb{R}^{2n}$, then for the weighted-particle model,
\begin{equation}\label{eq.wp_variance_estimation}
\textup{Var}\left({\Pi}^\mathrm{w} (s \cdot \mathrm{X}_t^\mathrm{w})\right) \le 2 M_s^2 \Vert f_I \Vert_\infty (1 + (K_V+\frac{\breve{\xi}^2}{\gamma_0})(T-t))\me^{2\max(K_V, \frac{\breve{\xi}^2}{\gamma_0})(T-t)} \Vert \varphi_T \Vert_2^2,
\end{equation} 
and for the signed-particle model, it has that
\begin{equation}\label{eq.sp_variance_estimation}
\textup{Var}\left({\Pi}^\mathrm{s} (s \cdot \mathrm{X}_t^\mathrm{w})\right) \le 2M_s^2 \Vert f_I \Vert_\infty (1+ (K_V+\frac{\breve{\xi}^2}{\gamma_0})(T-t))\me^{2\breve{\xi}(T-t)} \Vert \varphi_T \Vert_2^2.
\end{equation} 
\end{theorem}

\begin{proof}
It starts by a direct calculation
\begin{equation*}
\textup{Var}\left({\Pi}^\mathrm{w} (s \cdot \mathrm{X}_t^\mathrm{w})\right)  =  \lambda_0 \otimes\Pi_Q^\mathrm{w} (s^2 \cdot (\mathrm{X}_t^\mathrm{w})^2) - (\lambda_0 \otimes\Pi_Q^\mathrm{w}(s \cdot \mathrm{X}_t^\mathrm{w}))^2.
\end{equation*}
The first term reads
\begin{equation*}
\begin{split}
\lambda_0 \otimes\Pi_Q^\mathrm{w}(s^2 \cdot (\mathrm{X}_t^\mathrm{w})^2) = & \int_{\mathbb{R}^{2n}}  \frac{f_0^2(\bx, \bk)}{f_I(\bx, \bk)} \left(\int_{\Omega} (\mathrm{X}_t^\mathrm{w})^2(\omega)\mathrm{\Pi}^\mathrm{w}_{Q}(\D \omega) - \varphi^2(\bx, \bk, t) \right) \D \bx \D \bk \\
&+  \int_{\mathbb{R}^{2n} } \frac{f_0^2(\bx, \bk) }{f_I(\bx, \bk) } \varphi^2(\bx, \bk, t) \D \bx \D \bk,
\end{split}
\end{equation*}
and the second term is
\begin{equation*}
\lambda_0\otimes\Pi_Q^\mathrm{w}(s \cdot \mathrm{X}_t^\mathrm{w}) = \int_{\mathbb{R}^{n} \times \mathbb{R}^{n}} f_0(\bx, \bk) \varphi(\bx, \bk, t) \D \bx \D \bk.
\end{equation*}

Due to {H\"{o}lder's inequality}, it has that
\begin{equation*}
 \int_{\mathbb{R}^{n} \times \mathbb{R}^{n}} \frac{f_0^2(\bx, \bk)}{f_I(\bx, \bk) } \varphi^2(\bx, \bk, t) \D \bx \D \bk \le M_s^2 \Vert f_I \Vert_\infty \Vert \varphi(t) \Vert_2^2 \le M_s^2 \Vert f_I \Vert_\infty  (\me^{2K_V (T-t)} \Vert \varphi_T \Vert_2^2),
 \end{equation*}
so that 
\begin{equation*}
\textup{Var}\left({\Pi}^\mathrm{w} (s \cdot \mathrm{X}_t^\mathrm{w})\right) \le M_s^2 \Vert f_I \Vert_\infty  (\Vert {\Delta \phi_\mathrm{w}^{(2)}(t)} \Vert_1 + \me^{2K_V (T-t)} \Vert \varphi_T \Vert_2^2).
\end{equation*}
Thus according to Theorem \ref{thm.variance_estimation}, it yields Eq.~\eqref{eq.wp_variance_estimation}. 

{The proof} of Eq.~\eqref{eq.sp_variance_estimation} is similar and omitted for brevity.
\end{proof}

Theorem \ref{thm.variance_estimation} points out the {``numerical sign problem''} in the stochastic Wigner simulation \cite{SellierNedjalkovDimov2014, ShaoXiong2019, MuscatoWagner2016,XiongShao2019}, namely, the negative weights induced by $V_W = V_W^+ - V_W^-$ results in the exponential growth of the statistical errors, as well as the simulation cost along with the growth of particle number. However, the weighted-particle implementation allows the reduction of variance by increasing the parameter $\gamma_0$, thereby ameliorating the ``sign problem''. By contrast, the signed-particle implementation, although {with lower computational costs}, actually sacrifices the accuracy to some extent.

\section{The WBRW-SPA model}
\label{sec:WBRW-SPA}

Until now we have analyzed a class of branching random walk models based on HJD \eqref{eq.PDO_splitting}. The potential weakness of HJD lies in the fact that the near-cancelation of positive and negative parts of the oscillatory integral is totally neglected, leading to a rapid growth of variance in the related stochastic models. In this section, we try to formulate a new class of branching random walk models based on $\Theta^{\lambda_0}_V$, instead of $\Theta_V$. Intuitively speaking, we would like to {capture} the major contribution from the leading term of  the asymptotic expansion and throw away the high-order terms that may contribute less to the oscillatory integrals. The main result of this section is presented in Theorem \ref{SPA_thm.variance_estimate}. It is presented that the resulting stochastic processes gain a substantial {reduction} in variance, at the cost of introducing some biases.

According to Eqs.~\eqref{def.critical_point}, it is readily to verify that $\sigma_+(x) = -\sigma_-(x)$. Therefore, as $\psi(\bk, t) = \psi^\ast(-\bk, t)$, it has that $\psi(r\sigma_+, t) =\psi^\ast(r\sigma_-, t)$ and then the imaginary part of the stationary phase approximation vanishes, say, 
\begin{equation}\label{eq.PDO_real_property}
\begin{split}
&\Lambda_+^{>\lambda_0}[\varphi](\bx, \bk, t) +\Lambda_-^{>\lambda_0}[\varphi](\bx, \bk, t) \\
&= 2\int_{{\lambda_0}/{|\bz(\bx)|}}^{+\infty} \textup{Im}\left[\me^{ \mi r |\bz(\bx)|} \left(\frac{2\pi}{\mi r |\bz(\bx)|}\right)^{\frac{n-1}{2}} r^{n-1} \psi(r\sigma_+, t)\right] \Delta_{r\sigma_+} [\varphi](\bx, \bk, t) \D r \\
& = 2\int_{{\lambda_0}/{|\bz(\bx)|}}^{+\infty} \textup{Im}\left[ \zeta(r, \bx, t)\right] \cdot  r^{n-1}\Psi(r)\chi_{0, R}(r) \cdot \Delta_{r\sigma_+} [\varphi](\bx, \bk, t) \D r,
\end{split}
\end{equation}
where $\textup{Im}[z]$ denotes the imaginary part of $z$ and $\zeta(r, \bx, t)$ is given by that
\begin{equation}
\zeta(r, \bx, t) =  \me^{\mi r |\bz(\bx)|} \left(\frac{2\pi}{ \mi r |\bz(\bx)|}\right)^{\frac{n-1}{2}}  \frac{\psi(r\sigma_+, t)}{\Psi(r)}.
\end{equation}
Under the assumption {({\textbf{A3}})}, we have that
\begin{equation}
\breve{\eta} = \int_{0}^{+\infty} r^{n-1} \Psi(r) \chi_{0, R}(r) \D r  < \infty.
\end{equation}

The corresponding probability laws $\mathsf{\Pi}^\mathrm{w}_{Q}$ and $\mathsf{\Pi}^\mathrm{s}_{Q}$ are characterized in a similar pattern as in Eq.~\eqref{def.probability_law}, with an alternative setting of the transition density $\tilde{\mathrm{K}}_{i_1 \cdots i_m}^{t^{\prime}, \bx^{\prime}}(\bk; \bk^{\prime})$ and the particle weight $w_{i_1 \cdots i_m}$.  The definition of  $\tilde{\mathrm{K}}_{i_1 \cdots i_m}^{t^{\prime}, \bx^{\prime}}(\bk; \bk^{\prime})$ is given by that 
\begin{equation}
\tilde{\mathrm{K}}_{i_1 \cdots i_m}^{t^{\prime}, \bx^{\prime}}(\bk; \bk^{\prime}) =  \left\{
\begin{split}
&\frac{V_{W, R}^+(\bx^{\prime}, (-1)^{i_m}(\bk - \bk^{\prime}), t^{\prime})}{\xi(\bx^{\prime}, t^{\prime})}, &\quad &i_m = 1,2,\\
&\frac{r^{n-1} \Psi(r) \chi_{0, R}(r)}{\breve{\eta}} \cdot \delta(\sigma - \sigma_+(\bx^{\prime})) \Big |_{\bk = \bk^{\prime} + \frac{(-1)^{i_m}}{2}r\sigma } , &\quad &i_m = 3, 4, \\
& \delta(\bk - \bk^{\prime}), &\quad &i_m = 5,
\end{split}
\right.
\end{equation}
in the sense that for $i_m = 3$ or $4$, 
\begin{equation*}
\begin{split}
&\Pr(\bk_{i_1 \cdots i_m} - \bk_{i_1 \cdots i_{m-1}}\in \frac{\mathcal{K}_{i_m}}{2})  = \int_{\mathcal{K}_{i_m}} \D \bk_{i_1 \cdots i_m} ~ \tilde{\mathrm{K}}^{i_1 \cdots i_m}_{t_{i_1 \cdots i_m}, \bx_{i_1 \cdots i_m}}(\bk_{i_1 \cdots i_{m}}; \bk_{i_1 \cdots i_{m-1}}) \\
& = \int_{E_{i_m}} \frac{r_{i_1 \cdots i_m}^{n-1}\Psi(r_{i_1 \cdots i_m}) \chi_{0, R}(r_{i_1 \cdots i_m})}{ \breve{\eta}}  \cdot \delta(\sigma_{i_1 \cdots i_m} - \sigma_+(\bx_{i_1 \cdots i_m})) \D r_{i_1 \cdots i_m} \D \sigma_{i_1 \cdots i_m}, 
\end{split}
\end{equation*}
where $E_{i_m}$ is the conversion of $\mathcal{K}_{i_m}$ under the spherical coordinate. 

For brevity, we adopt the following notations: $r_{i_1\cdots i_m} = 2 \left| \bk_{i_1\cdots i_m} - \bk_{i_1 \cdots i_{m-1}} \right| $, $\zeta_{i_1 \cdots i_m} = \zeta( r_{i_1\cdots i_m}, \bx_{i_1\cdots i_m}, t_{i_1\cdots i_m})$ and $\bz_{i_1 \cdots i_m} = \bz(\bx_{i_1 \cdots i_m})$. Let $\sigma_{i_1 \cdots i_m}$ be
\begin{equation}
\sigma_{i_1 \cdots i_m} = \left\{
\begin{split}
&\frac{\xi(\bx_{i_1\cdots i_m}, t_{i_1\cdots i_m})}{\gamma_0} \cdot \mone_{\{r_{i_1\cdots i_m} \le \frac{2 \lambda_0}{|\bz_{i_1\cdots i_m}|}\}}, &i_m = 1, 2,\\
&\frac{2\breve{\eta}}{\gamma_0} \cdot \big| \textup{Im} \left[ \zeta_{i_1 \cdots i_m} \right] \big|  \cdot \mone_{\{r_{i_1\cdots i_m} > \frac{\lambda_0}{|\bz_{i_1\cdots i_m}|}\}}, &i_m = 3, 4. 
\end{split}
\right.
\end{equation}
Then the particle weight $w_{i_1\cdots i_m}$ reads that
\begin{itemize}

\item[(1)] For the weighted particle model, 
\begin{equation}\label{def_wbrw_spa_w}
w_{i_1\cdots i_m} = \left\{
\begin{split}
&\sigma_{i_1 \cdots i_m} \cdot  \mone_{\mathcal{K}}(\bk_{i_1\cdots i_m}),  \quad &i_m &= 1, 2,\\
&\sigma_{i_1 \cdots i_m} \cdot \textup{sgn}\left(\textup{Im}[\zeta_{i_1\cdots i_m}]\right)  \cdot \mone_{\mathcal{K}}(\bk_{i_1\cdots i_m}), \quad &i_m &= 3, 4, \\
&1,  \quad &i_m &= 5.
\end{split}
\right.
\end{equation}

\item[(2)] For the signed particle model, 
\begin{equation}\label{def_wbrw_spa_s}
w_{i_1\cdots i_m} = \left\{
\begin{split}
&1,  &\textup{with} ~\Pr = \sigma_{i_1 \cdots i_m} \cdot \mone_{\mathcal{K}}(\bk_{i_1\cdots i_m}),  ~~i_m = 1, 2,\\
&\textup{sgn}\left(\textup{Im}[\zeta_{i_1\cdots i_m}]\right), &\textup{with} ~\Pr = \sigma_{i_1 \cdots i_m} \cdot \mone_{\mathcal{K}}(\bk_{i_1\cdots i_m}), ~~ i_m = 3, 4, \\
&0,   &\textup{otherwise} ,
\end{split}
\right.
\end{equation}
and $w_{i_1\cdots i_m} = 1$ for $i_m = 5$.

\end{itemize}
Here $\textup{sgn}(x) = 1$ for $x > 0$, $\textup{sgn}(x) = -1$ for $x < 0$ and $\textup{sgn}(x) = 0$ for $x = 0$.

\begin{definition}\label{def.stochastic_process_SPA}
Suppose $(\bx_i, \bk_i)$ is the starting state of a frozen particle $i$ in a given family history $\omega$, and let $\delta_{(\bx, \bk)}$ be the Dirac measure concentrated at  state $(\bx, \bk)$. Then the weighted-particle WBRW-SPA is given by that
\begin{equation}\label{def.wpWBRW_SPA}
\mathrm{Y}_t^\mathrm{w}(\omega) = \langle \varphi_T, \sum_{i \in \mathcal{E}_t(\omega)} \tilde{w}_{i} \cdot \delta_{(\bx_{i}(T-t_i), \bk_{i})} \rangle =  \sum_{i \in \mathcal{E}_t(\omega)} \tilde{w}_i \cdot \varphi_T(\bx_i(T-t_i), \bk_i),
\end{equation}
the cumulative weight $\tilde{w}_i \in [-1, 1]$ for $i = \langle i_1 i_2 \cdots i_n \rangle$ is defined by the product of the particle weights $w_{i_1\cdots i_m}$ 
\begin{equation}
\tilde{w}_{i} = \prod_{m=1}^n  (-1)^{i_m +1} w_{i_1\cdots i_m} , \quad |w_{i_1\cdots i_m}| \le 1,
\end{equation}
where $w_{i_1\cdots i_m}$ are defined in Eq.~\eqref{def_wbrw_spa_w}.

Similarly, the signed-particle WBRW-SPA is given by that
\begin{equation}\label{def.spWBRW_SPA}
\mathrm{Y}_t^\mathrm{s}(\omega) = \langle \varphi_T, \sum_{i \in \mathcal{E}_t(\omega)} \tilde{s}_{i} \cdot \delta_{(\bx_{i}(T-t_i), \bk_{i})} \rangle =  \sum_{i \in \mathcal{E}_t(\omega)} \tilde{s}_i \cdot \varphi_T(\bx_i(T-t_i), \bk_i),
\end{equation}
where cumulative weight $\tilde{s}_i \in \{-1, 0, 1\}$ for $i = \langle i_1 i_2 \cdots i_n \rangle$ is defined by the product of the particle weight $w_{i_1\cdots i_m}$,
\begin{equation}
\tilde{s}_i =  \prod_{m=1}^n (-1)^{i_m+1} w_{i_1\cdots i_m}, \quad w_{i_1\cdots i_m} = \{-1,0, 1\},
\end{equation}
where $w_{i_1\cdots i_m}$ are defined in Eq.~\eqref{def_wbrw_spa_s}.

\end{definition}

\begin{definition}
The first moments of $\mathrm{Y}_t^\mathrm{w}$ and $\mathrm{Y}_t^\mathrm{s}$ are denoted by
\begin{equation}
\Phi_\mathrm{w}^{(1)}(\bx, \bk, t) = \mathsf{\Pi}^\mathrm{w}_{Q} \mathrm{Y}_t^\mathrm{w}, \quad \Phi_\mathrm{s}^{(1)}(\bx, \bk, t) = \mathsf{\Pi}^\mathrm{s}_{Q} \mathrm{Y}_t^\mathrm{s}, 
\end{equation}
and the second moments are
\begin{equation}
\Phi_\mathrm{w}^{(2)}(\bx, \bk, t) = \mathsf{\Pi}^\mathrm{w}_{Q} (\mathrm{Y}_t^\mathrm{w})^2, \quad \Phi_\mathrm{s}^{(2)}(\bx, \bk, t) = \mathsf{\Pi}^\mathrm{s}_{Q} (\mathrm{Y}_t^\mathrm{s})^2, 
\end{equation}
In addition, the variances are defined as
\begin{align}\label{def.Yt_second_moment}
& \Delta \Phi^{(2)}_\mathrm{w}(\bx, \bk, t)=  \mathsf{\Pi}^\mathrm{w}_{Q} (\mathrm{Y}_t^\mathrm{w} - \Phi_\mathrm{w}^{(1)}(\bx, \bk, t))^2,\\
& \Delta \Phi^{(2)}_\mathrm{s}(\bx, \bk, t) =  \mathsf{\Pi}^\mathrm{s}_{Q} (\mathrm{Y}_t^\mathrm{s} -{\Phi_\mathrm{s}^{(1)}(\bx, \bk, t))^2}.
\end{align}
\end{definition}

Now we sketch the proof of Theorem \ref{SPA_thm.variance_estimate}. The first step is to prove that both $\Phi_\mathrm{w}^{(1)}(\bx, \bk, t)$ and $\Phi_\mathrm{s}^{(1)}(\bx, \bk, t)$ are solutions of the modified backward Wigner equation,
\begin{equation}\label{eq.modified_backward_Wigner}
\left\{
\begin{split}
&\frac{\partial }{\partial t}\varphi_{\lambda_0}(\bx, \bk, t) + \frac{\hbar \bk}{m} \cdot \nabla_{\bx} \varphi_{\lambda_0}(\bx, \bk, t) = \Theta^{\lambda_0}_{V}\left[\varphi_{\lambda_0}\right](\bx, \bk, t), ~~t \le T,\\
& \varphi_{\lambda_0}(\bx, \bk, T) = \varphi_T(\bx, \bk).
\end{split}
\right.
\end{equation}
The second step is to estimate the variances of the resulting stochastic models. Finally, by comparing $\varphi$ and $\varphi_{\lambda_0}$, we arrive at the final result. 

\begin{theorem}[The first moments of WBRW-SPA]\label{thm_first_moment_wbrw_spa}
Suppose the assumptions ${\bf (A3)}$ are satisfied. Then for a fixed $\lambda_0$, there exists $\varphi_{\lambda_0} \in C([0, T],  L^2(\mathbb{R}^n) \times L^2_0(\mathbb{R}^n))$ such that
\begin{equation}
\Phi_\mathrm{w}^{(1)}(\bx, \bk, t) = \Phi_\mathrm{w}^{(2)}(\bx, \bk, t) = \varphi_{\lambda_0}(\bx, \bk, t),
\end{equation}
and $\varphi_{\lambda_0}$ satisfies the following {estimate:}
\begin{equation}\label{spa_bias_estimate}
\Vert \varphi(t) - \varphi_{\lambda_0}(t) \Vert_2 \le C \lambda_0^{-{n}/{2}} \max_{t \in [0, T]} \Vert \varphi(t) \Vert_{L_{\bx}^2 \times H_{\bk}^1}.
\end{equation}
\end{theorem}

\begin{proof}
Consider the events $\mathrm{E}_t = \{\tau_0: t+\tau_0 \ge T \}$ and $\mathrm{E}_t^c =  \{\tau: t+\tau_0 < T \}$. 
\begin{equation}
\mathsf{\Pi}_Q^\mathrm{w}\left(\mone_{\mathrm{E}_t}(\omega)\right) = \me^{-\gamma_0(T-t)}\varphi_T(\bx(T-t), k),
\end{equation}
Suppose the event $\mathrm{E}_t^c$ occurs, then for the family history $\omega = (Q_0; \omega_1, \cdots, \omega_5)$, it has that
\begin{equation}\label{eq.expansion_Y_t_w}
\mathrm{Y}_t^\mathrm{w}(\omega) = \sum_{i=1}^4 (-1)^{i+1}\zeta_{i} \cdot \mathrm{Y}_t^\mathrm{w}(\omega_i) + \mathrm{Y}_t^\mathrm{w}(\omega_5).
\end{equation}
Now we calculate $\mathrm{\Pi}_Q^\mathrm{w} \mathrm{Y}_t^\mathrm{w}$ for the first family $\omega_1$. Owing to the fact that $r_i$ are independent, for the Hahn-Jordan decomposition $\Lambda^{<\lambda_0} = \Lambda_+^{<\lambda_0} - \Lambda_-^{<\lambda_0}$, it has that
\begin{equation}\label{eq.spa_first_stochastic_representation}
\int_0^{T-t}  \frac{\Lambda_-^{<\lambda_0}[\Phi_\mathrm{w}^{(1)}](\bx(\tau), \bk, t+\tau)}{\gamma_0} \D \mathcal{G}(\tau) = \int_{\mathrm{E}_t^c} (-1)^{1+1} \zeta_1 \cdot \mathrm{Y}_t^\mathrm{w}(\omega_1) \mathrm{\Pi}_Q^\mathrm{w}(\D \omega)
\end{equation}
since
\begin{equation*}
\begin{split}
& \int_{\mathrm{E}_t^c}  \frac{\xi(\bx(\tau), t + \tau)}{\gamma_0} \left( \int_{\Omega_1}\mone_{\{r_1 < \lambda_0/|{2\bz_1}|\} } \cdot \mathrm{Y}_t^\mathrm{w}(\omega_1) \mathsf{\Pi}^\mathrm{w}_{Q_1}(\D \omega_1) \right) \mathsf{\Pi}_Q^\mathrm{w}(\D \omega) = \\
&\int_0^{T-t} \D \mathcal{G}(\tau)\int_{B(\frac{\lambda_0}{|{2\bz_1}|})}\D \bk_1 \frac{\xi(\bx(\tau) , t+\tau)}{\gamma_0}  \frac{V^-_W(\bx(\tau), \bk_1, t+\tau)}{\xi(\bx(\tau), t+\tau)} \mathrm{\Phi}_\mathrm{w}^{(1)}(\bx(\tau), \bk-{\bk_1}, t+\tau) .
\end{split}
\end{equation*}
where $\bz_i$ is short for $\bz(\bx_i) = \bz(\bx(\tau))$. For the third family $\omega_3$, it yields that
\begin{equation}\label{eq.spa_third_stochastic_representation}
\int_0^{T-t}  \frac{\Lambda_-^{ > \lambda_0}[\Phi_\mathrm{w}^{(1)}](\bx(\tau), \bk, t+\tau)}{\gamma_0} \D \mathcal{G}(\tau) = \int_{\mathrm{E}_t^c} (-1)^{1+3} \zeta_3 \cdot \mathrm{Y}_t^\mathrm{w}(\omega_3) \mathrm{\Pi}_Q^\mathrm{w}(\D \omega)
\end{equation}
since
\begin{equation*}
\begin{split}
&\int_{ \mathrm{E}_t^c} (-1)^{1+3} \zeta_3 \cdot \mathrm{Y}_t^\mathrm{w}(\omega_3) \mathsf{\Pi}_Q^\mathrm{w}(\D \omega) \\
&=\int_{\mathrm{E}_t^c}   \frac{2\breve{\eta} }{\gamma_0}\left(\int_{\Omega_3} \textup{Im}[\sigma_3] \cdot \mone_{\{r_3 > \lambda_0/|\bz_3|\}} \cdot \mathrm{Y}_t^\mathrm{w}(\omega_3)  \mathsf{\Pi}_{Q_3}^\mathrm{w}(\D \omega_3) \right) \mathsf{\Pi}_Q^\mathrm{w}(\D \omega)\\
&=\int_0^{T-t} \D \mathcal{G}(\tau) \int_{{\lambda_0}/{|\bz_3|}}^{+\infty} \D r_3 ~  \frac{2\breve{\eta}}{\gamma_0} \cdot \textup{Im}\left[ \me^{\mi r_{3}|\bz_3|} \left(\frac{2\pi}{\mi r_{3} |\bz_3|}\right)^{\frac{n-1}{2}}\frac{\psi(r_3\sigma_+(\bx(\tau)), t+\tau)}{\Psi(r_3)}\right] \\
& \quad \times \frac{r_3^{n-1}\Psi(r_3)}{ \breve{\eta}}\mathrm{\Phi}_\mathrm{w}^{(1)}(\bx(\tau), \bk - \frac{r_3 \sigma_+(\bx(\tau))}{2}, t+\tau)  .
\end{split}
\end{equation*}
The other terms are tackled in a similar pattern. Summing over five terms recovers $-\Theta_V^{\lambda_0}[\varphi] +\gamma_0 \cdot \varphi$. The proof for the signed-particle model is similar and thus omitted for brevity. 

For the second part, let $\varepsilon(\bx, \bk, t) = \varphi_{\lambda_0}(\bx, \bk, t) - \varphi(\bx, \bk, t)$. According to Eqs.\eqref{eq.backward_Wigner} and \eqref{eq.modified_backward_Wigner}, it is observed that $\varepsilon(\bx, \bk, t)$ satisfies the following equation
\begin{equation}\label{def_error_equation}
\frac{\partial }{\partial t}\varepsilon(\bx, \bk, t) + \frac{\hbar \bk}{m} \cdot \nabla_{\bm{x}}\varepsilon(\bx, \bk, t) = \Theta^{\lambda_0}_V[\varepsilon](\bx, \bk, t) + (\Theta_V - \Theta^{\lambda_0}_V)[\varphi](\bx, \bk, t). 
\end{equation}
with {$\varepsilon(\bx, \bk, T) = 0$}. It's easy to verify by Eq.~\eqref{eq.prior_L2_estimate} that
\begin{equation*}
\begin{split}
\Vert \varepsilon(t) \Vert_2 & \le  \int_t^T \D \mathcal{G}(t^{\prime} -t) \Vert \Theta_V[\varphi](t^{\prime}) - \Theta_V^{\lambda_0}[\varphi](t^{\prime}) \Vert_2 + (1+\frac{K_V}{\gamma_0}) \int_t^T \D \fG(t^{\prime}-t) \Vert \varepsilon(t^{\prime}) \Vert_2 \\
& \le (1 - \me^{-\gamma_0 (T-t)}) \max_{t \in [0, T]} \lambda_0^{-\frac{n}{2}} \Vert \varphi(t) \Vert_{L^2_{\bx} \times H^1_{\bk}} + (1+\frac{K_V}{\gamma_0}) \int_t^T \D \fG(t^{\prime}-t) \Vert \varepsilon(t^{\prime}) \Vert_2 \\
& \le \frac{\lambda_0^{-\frac{n}{2}} (\gamma_0 + K_V)\me^{K_V(T-t)}}{K_V} \max_{t \in [0, T]} \Vert \varphi(t) \Vert_{L^2_{\bx} \times H^1_{\bk}}
\end{split}
\end{equation*}
The second inequality uses the remainder estimate \eqref{eq.error_estimate_stationary_phase_method} in Theorem \ref{thm.stationary_phase_approximation} and the third equality is derived by Lemma \ref{lemma.Gronwall_ineq}.
\end{proof}

The next step is to estimate the upper bounds for the variances, as illustrated by the following theorem. When the assumption (\textbf{A4}) holds, by Young's inequality, the HJD of the operator $\Lambda^{<\lambda_0}$, denoted by $\Lambda^{<\lambda_0}_\pm$, is bounded from $L^p(\mathbb{R}^n) \times L_0^p(\mathbb{R}^n)$ to itself,
\begin{equation}\label{eq.PDO_plus_bound_improve}
\Vert \Lambda^{\lambda_0}_\pm [\varphi](t) \Vert_p \le \alpha_\ast \breve{\xi} \Vert \varphi(t) \Vert_p, \quad p = 1, 2.
\end{equation}

\begin{theorem}[Variances of WBRW-SPA]
Suppose the assumptions {{\bf (A2)}-{\bf (A4)} }are satisfied for a sufficiently large $\lambda_0$ and {$\gamma_2 = \alpha_\ast \breve{\xi}^2$}. Then we have that
\begin{equation} \label{eq.wbrw_spa_wp_variance_estimate_1}
\Vert  \mathsf{\Pi}_{Q}^\mathrm{w} (\mathrm{Y}_t^\mathrm{w} - \varphi_{\lambda_0}(t))^2 \Vert_1 \lesssim (1+\frac{4\gamma_2}{\gamma_0}(T-t)) \me^{2\max(K_V, \frac{\alpha_\ast \breve{\xi}^2}{\gamma_0}) (T-t)} \Vert \varphi_T \Vert_2^2 - \Vert \varphi_{\lambda_0}(t) \Vert_2^2 
\end{equation}
and
\begin{equation} \label{eq.wbrw_spa_sp_variance_estimate_1}
\Vert  \mathsf{\Pi}_{Q}^\mathrm{s} (\mathrm{Y}_t^\mathrm{s} - \varphi_{\lambda_0}(t))^2 \Vert_1 \lesssim (1+2(K_V+ \frac{\gamma_2}{\gamma_0})(T-t))\me^{2\alpha_\ast \breve{\xi}(T-t)} \Vert \varphi_T \Vert_2^2 - \Vert \varphi_{\lambda_0}(t) \Vert_2^2.
\end{equation}
\end{theorem}

\begin{proof}
According to Eqs.~\eqref{def_wbrw_spa_w} and \eqref{eq.expansion_Y_t_w}, the renewal-type equation for the second moment of $\mathrm{Y}_t^\mathrm{w}$ reads that
\begin{equation}
\mathsf{\Pi}^\mathrm{w}_{Q} (\mathrm{Y}_t^\mathrm{w})^2  =  \mathsf{\Pi}^\mathrm{w}_{Q} \left( \mone_{\mathrm{E}_t} (\mathrm{Y}_t ^\mathrm{w})^2 \right) + \mathsf{\Pi}^\mathrm{w}_{Q} \left( \mone_{\mathrm{E}_t^c} (\mathrm{Y}_t ^\mathrm{w})^2  \right),
\end{equation}
where the second term on the right hand side reads
\begin{equation*}
\begin{split}
&\mathsf{\Pi}^\mathrm{w}_{Q} \left( \mone_{\mathrm{E}_t^c} (\mathrm{Y}_t^\mathrm{w})^2  \right) = \sum_{i=1}^5  \int_{\mathrm{E}_t^c} \left( \int_{\Omega_i} w^2_i \cdot (\mathrm{Y}_{t+\tau_0}^\mathrm{w}(\omega_i))^2 \mathsf{\Pi}^\mathrm{w}_{Q_i}(\D \omega_i) \right) \mathsf{\Pi}^\mathrm{w}_{Q}(\D \omega) \\
&+\sum_{i \ne j}  \int_{\mathrm{E}_t^c} \left((-1)^{i+j}  \int_{\Omega_i} w_i   \mathrm{Y}^\mathrm{w}_{t+\tau_0}(\omega_i) \mathsf{\Pi}^\mathrm{w}_{Q_i}(\D \omega_i) \int_{\Omega_j} w_j \mathrm{Y}^\mathrm{w}_{t+\tau_0} (\omega_j) \mathsf{\Pi}^\mathrm{w}_{Q_j}(\D \omega_j)\right) \mathsf{\Pi}^\mathrm{w}_{Q}(\D \omega).
\end{split}
\end{equation*}

Now we define an operator $B_\mathrm{w}^{\lambda_0}$ and verify its $L^1$-boundedness,
\begin{equation}\label{def.B_w_lambda_0}
B_\mathrm{w}^{\lambda_0}[\Phi_\mathrm{w}^{(2)}](\bx(\tau_0), \bk, t+\tau_0) = \sum_{i=1}^5  \int_{\Omega_i}  w^2_i \cdot (\mathrm{Y}_{t+\tau_0}^\mathrm{w})^2(\omega_i) \mathsf{\Pi}^\mathrm{w}_{Q_i}(\D \omega_i) . 
\end{equation}
Since $|\textup{Im}[\zeta(r_3, \bx(\tau), t+\tau)] |^2 \le \left({2\pi}/{\lambda_0}\right)^{n-1}$, it has that
\begin{equation}\label{ineq.asymp_term}
\begin{split}
& \Vert \int_{ \mathrm{E}_t^c}  \left( \int_{\Omega_3} w^2_3 \cdot (\mathrm{Y}_{t+\tau_0}^\mathrm{w}(\omega_3) )^2\mathsf{\Pi}^\mathrm{w}_{Q_3}(\D \omega_3) \right) \mathsf{\Pi}^\mathrm{w}_{Q}(\D \omega) \Vert_{L^1_{\bk}} \\
& \le \frac{4\breve{\eta}}{\gamma_0^2}  (\frac{2\pi}{\lambda_0})^{n-1} \Vert \int_{\frac{\lambda_0}{|\bz_3|}}^{+\infty} r_3^{n-1} \Psi(r_3) \chi_{0, R}(r_3) \cdot \Phi_\mathrm{w}^{(2)}(\bx(\tau), \bk - \frac{r_3\sigma_+(\bx(\tau))}{2}, t+\tau) \D r_3  \Vert_{L^1_{\bk}} \\
& \lesssim \frac{4\breve{\eta}^2}{\gamma_0^2}  (\frac{2\pi}{\lambda_0})^{n-1}  \Vert \Phi_\mathrm{w}^{(2)}(\bx(\tau), \bk, t +\tau) \Vert_{L^1_{\bk}}.
\end{split}
\end{equation}


Combining Eqs.~\eqref{eq.PDO_plus_bound_improve} and \eqref{ineq.asymp_term}, we obtain that
\begin{equation}
\Vert B_\mathrm{w}^{\lambda_0}[\Phi_\mathrm{w}^{(2)}](t) \Vert_1 \le (1+ \frac{2 \alpha_\ast\breve{\xi}^2}{\gamma_0^2}   + \frac{8\breve{\eta}^2}{\gamma_0^2}  (\frac{2\pi}{\lambda_0})^{n-1}) \Vert \Phi_\mathrm{w}^{(2)}(t) \Vert_1.
\end{equation}
Thus for a sufficiently large $\lambda_0$ (such as $\breve{\eta}^2 (\frac{2\pi}{\lambda_0})^{n-1} \ll \breve{\xi}^2)$, it further yields that
\begin{equation}\label{eq_Bw_estimate_asymptotic}
\Vert B_\mathrm{w}^{\lambda_0}[\Phi_\mathrm{w}^{(2)}](t) \Vert_1  \lesssim (1 + \frac{2\alpha_\ast \breve{\xi}^2}{\gamma_0^2}) \Vert \Phi_\mathrm{w}^{(2)}(t) \Vert_1 + \mathcal{O}(\lambda_0^{-(n-1)}),
\end{equation}

For the non-diagonal terms, we define correlated terms $C_{ij}^{\lambda_0}(\bx, \bk, t)$ as
\begin{equation}
C_{ij}^{\lambda_0}(\bx(\tau_0), \bk, t+\tau_0) =  (-1)^{i+j}  \int_{\Omega_i} w_i \mathrm{Y}^\mathrm{w}_{t+\tau_0}(\omega_i) \mathsf{\Pi}^\mathrm{w}_{Q_i}(\D \omega_i) \int_{\Omega_j} w_j   \mathrm{Y}^\mathrm{w}_{t+\tau_0} (\omega_j) \mathsf{\Pi}^\mathrm{w}_{Q_j}(\D \omega_j),
\end{equation}
in which each term can be calculated in the similar way as in Eqs.~\eqref{eq.spa_first_stochastic_representation} and \eqref{eq.spa_third_stochastic_representation}. In fact, by using Young's inequality and Cauchy-Schwarz inequality, $C^{\lambda_0}(\bx, \bk, t) = \sum_{i < j}C_{ij}^{\lambda_0}(\bx, \bk, t)$ can be estimated by
\begin{equation}\label{eq_C_estimate_asymptotic}
\begin{split}
\Vert C^{\lambda_0}(t) \Vert_1 \le &  \frac{\Vert \Lambda_+^{< \lambda_0}[\varphi_{\lambda_0}](t)\Vert_2  \cdot \Vert \Lambda_-^{<\lambda_0}[\varphi_{\lambda_0}](t)\Vert_2 }{\gamma_0^2} + \frac{\Vert \Lambda_+^{>\lambda_0}[\varphi_{\lambda_0}](t)\Vert_2  \cdot \Vert \Lambda_-^{>\lambda_0}[\varphi_{\lambda_0}](t)\Vert_2 }{\gamma_0^2} \\
&+ \frac{\Vert \Lambda_+^{< \lambda_0}[\varphi_{\lambda_0}](t) - \Lambda_-^{< \lambda_0}[\varphi_{\lambda_0}](t) \Vert_2 \cdot \Vert \Lambda_-^{>\lambda_0}[\varphi_{\lambda_0}](t)\Vert_2 }{\gamma_0^2} \\
&+ \frac{\Vert \Lambda_+^{< \lambda_0}[\varphi_{\lambda_0}](t) - \Lambda_-^{< \lambda_0}[\varphi_{\lambda_0}](t) \Vert_2  \cdot\Vert \Lambda_+^{>\lambda_0}[\varphi_{\lambda_0}](t)\Vert_2 }{\gamma_0^2} \\
& + \frac{\Vert \Lambda_+^{>\lambda_0}[\varphi_{\lambda_0}](t) - \Lambda_-^{>\lambda_0}[\varphi_{\lambda_0}](t)\Vert_2  \cdot\Vert \varphi_{\lambda_0}(t) \Vert_2^2 }{\gamma_0}  \\
& + \frac{\Vert \Lambda_+^{<\lambda_0}[\varphi_{\lambda_0}](t) - \Lambda_-^{<\lambda_0}[\varphi_{\lambda_0}](t)\Vert_2  \cdot\Vert \varphi_{\lambda_0}(t) \Vert_2 }{\gamma_0}  \\
\le & \left(\frac{\alpha_\ast^2 \breve{\xi}^2}{\gamma_0^2} + \left(\frac{4\breve{\eta}^2}{\gamma_0^2}(\frac{2\pi}{\lambda_0})^{\frac{n-1}{2}} + \frac{4K_V \breve{\eta}}{\gamma_0^2} + \frac{4\breve{\eta}}{\gamma_0}\right) (\frac{2\pi}{\lambda_0})^{\frac{n-1}{2}} + \frac{K_V}{\gamma_0}\right) \Vert \varphi_{\lambda_0}(t) \Vert_2^2 \\
\lesssim & (\frac{ \alpha_{\ast}\breve{\xi}^2}{\gamma_0^2} + \frac{ K_V}{\gamma_0}) \Vert \varphi_{\lambda_0}(t) \Vert_2^2 + \mathcal{O}(\lambda_0^{-\frac{n-1}{2}}).
\end{split}
\end{equation}


In addition, since 
\begin{equation}
\Vert \Theta_V^{\lambda_0}[\varphi_{\lambda_0}](t) \Vert_2 \le K_V \Vert \varphi_{\lambda_0}(t) \Vert_2 + C\lambda_0^{-\frac{n-1}{2}} \Vert \varphi_{\lambda_0}(t) \Vert_2,
\end{equation}
according to Lemma \ref{lemma.L2_estimate}, it yields that
\begin{equation}\label{eq_L2_estimate_asymptotic}
\Vert \varphi_{\lambda_0}(t) \Vert_2  \lesssim \me^{K_V(T-t)} \Vert \varphi_T \Vert_2 + \mathcal{O}(\lambda_0^{-\frac{n-1}{2}}).
\end{equation}

Combining Eqs.~\eqref{eq_Bw_estimate_asymptotic}, \eqref{eq_C_estimate_asymptotic} and \eqref{eq_L2_estimate_asymptotic}, we have the following result
\begin{equation}\label{eq.SPA_variation_inequality}
\begin{split}
\Vert \Delta \Phi_\mathrm{w}^{(2)}(t) \Vert_1 
 &\lesssim \me^{-\gamma_0(T-t)}  \Vert \varphi_T \Vert^2_2 - \Vert \varphi_{\lambda_0}(t) \Vert_2^2 +2 \int_t^T \D \mathcal{G}(t^{\prime}-t) \Vert C^{\lambda_0}(t^{\prime}) \Vert_1\\
 &+ (1+\frac{2 \alpha_\ast \breve{\xi}^2}{\gamma_0^2}) \int_t^T \D \mathcal{G}(t^{\prime}-t) \left\{ \Vert \Delta \Phi_\mathrm{w}^{(2)}(t^{\prime}) \Vert_1 + \Vert \varphi_{\lambda_0}(t^{\prime}) \Vert_2 \right\} + \mathcal{O}(\lambda_0^{-\frac{n-1}{2}}).
\end{split} 
\end{equation}

Notice that by replacing $\breve{\xi}$ in Eqs.~\eqref{eq.variation_inequality} and \eqref{key_estimate_ut} by $\sqrt{\alpha_\ast} \breve{\xi}$, we can further obtain the following estimate for Eq.~\eqref{eq.SPA_variation_inequality}, 
\begin{equation}\label{eq.wp_spa_variation_L1_bound}
\begin{split}
\Vert \Delta \Phi_\mathrm{w}^{(2)}(t) \Vert_1 \lesssim &\left(\frac{K_V \gamma_0 + \alpha_\ast \breve{\xi}^2}{K_V \gamma_0 - \alpha_\ast \breve{\xi}^2} \me^{2K_V(T-t)} - \frac{2\alpha_\ast \breve{\xi}^2}{K_V\gamma_0 -\alpha_\ast \breve{\xi}^2} \me^{\frac{2\alpha_\ast \breve{\xi}^2}{\gamma_0}(T-t)}\right) \Vert \varphi_T\Vert_2^2 \\
&- \Vert \varphi_{\lambda_0}(t) \Vert_2^2 + \mathcal{O}(\lambda_0^{-\frac{n-1}{2}}).
\end{split}
\end{equation}
Therefore, for the weighted particle model,  it has that
\begin{equation}\label{eq.bound_wp_spa_variation}
\Vert \Delta \Phi_\mathrm{w}^{(2)}(t) \Vert_1 \lesssim (1+\frac{4\alpha_\ast \breve{\xi}^2}{\gamma_0}(T-t))\me^{2\max(K_V, \frac{\alpha_ \ast \breve{\xi}^2}{\gamma_0}) (T-t)} \Vert \varphi_T \Vert_2^2 - \Vert \varphi_{\lambda_0}(t) \Vert_2^2 + \mathcal{O}(\lambda_0^{-\frac{n-1}{2}}).
\end{equation}

Similarly, we can obtain the estimate of the variance for the signed-particle model, 
\begin{equation}
\Vert \Delta \Phi_\mathrm{s}^{(2)}(t) \Vert_1 \lesssim (1+(2K_V+ \frac{2 \alpha_\ast \breve{\xi}^2}{\gamma_0})(T-t))\me^{2\alpha_\ast\breve{\xi}(T-t)} \Vert \varphi_T \Vert_2^2 - \Vert \varphi_{\lambda_0}(t) \Vert_2^2 + \mathcal{O}(\lambda_0^{-\frac{n-1}{2}}).
\end{equation}
When $\lambda_0$ is sufficiently large, we can throw way the remainder term $\mathcal{O}(\lambda_0^{-\frac{n-1}{2}})$ and arrive at Eqs.~\eqref{eq.wbrw_spa_wp_variance_estimate_1} and \eqref{eq.wbrw_spa_sp_variance_estimate_1}.
\end{proof}

As the last step, we complete the proof of Theorem \ref{SPA_thm.variance_estimate}.

\begin{proof}[Proof of Theorem \ref{SPA_thm.variance_estimate}]
Since
\begin{equation*} 
\begin{split}
\mathsf{\Pi}_Q^\mathrm{w} (\mathrm{Y}_t^\mathrm{w} - \varphi(\bx, \bk, t))^2 =& \mathsf{\Pi}_Q^\mathrm{w} (\mathrm{Y}_t^\mathrm{w} - \varphi_{\lambda_0}(\bx, \bk, t) + \varphi_{\lambda_0}(\bx, \bk, t) - \varphi(\bx, \bk, t))^2 \\
 =& \mathsf{\Pi}_Q^\mathrm{w} (\mathrm{Y}_t^\mathrm{w} - \varphi_{\lambda_0}(\bx, \bk, t))^2 + (\varphi_{\lambda_0}(\bx, \bk, t) - \varphi(\bx, \bk, t))^2\\
&+ 2 \mathsf{\Pi}_Q^\mathrm{w}(\mathrm{Y}_t^\mathrm{w} - \varphi_{\lambda_0}(\bx, \bk, t)) \cdot (\varphi_{\lambda_0}(\bx, \bk, t) - \varphi(\bx, \bk, t)).
\end{split} 
\end{equation*}
By the {extended Minkowski's inequality and the Cauchy-Schwarz inequality}, it has that 
\begin{equation*} \label{eq.triangular_inequality}
\begin{split}
\Vert \mathsf{\Pi}_Q^\mathrm{w} (\mathrm{Y}_t^\mathrm{w} - \varphi(\bx, \bk, t))^2 \Vert_1 \le & \Vert  \mathsf{\Pi}_Q^\mathrm{w} (\mathrm{Y}_t^\mathrm{w} - \varphi_{\lambda_0}(\bx, \bk, t))^2 \Vert_1 + \Vert \varphi_{\lambda_0}(t) - \varphi(t) \Vert_2^2\\
&+ 2 \Vert \mathsf{\Pi}_Q^\mathrm{w}(\mathrm{Y}_t^\mathrm{w} - \varphi_{\lambda_0}(t))\Vert_2 \cdot \Vert \varphi_{\lambda_0}(t) - \varphi(t))\Vert_2.
\end{split} 
\end{equation*}

The first term is bounded by Eq.~\eqref{eq.bound_wp_spa_variation}. The estimate $\Vert \varphi_{\lambda_0}(t) - \varphi(t) \Vert_2^2$ in the second term is given by Eq.~\eqref{spa_bias_estimate}. And the third term is zero due to Theorem \ref{thm_first_moment_wbrw_spa}. Thus it finalizes the proof of Eq.~\eqref{eq.wbrw_spa_wp_variance_estimate}. The proof of Eq.~\eqref{eq.wbrw_spa_sp_variance_estimate} is similar and omitted here for brevity.
\end{proof}


Theorem \ref{SPA_thm.variance_estimate} implies that the variance of the inner product problem can be reduced by chopping the $\bx$-support and adopting the WBRW-SPA for the region in which the distance $|\bz(\bx)|$ between two interacting bodies is sufficiently large (as stated in the assumption {\bf (A4)}). The price to pay is to introduce some asymptotic biases, which may be negligible when $\alpha_\ast$ is sufficiently small. In Section \ref{sec:num}, we will show that all the results in our theoretical analysis can be {verified} in numerical experiments.

\section{Numerical validation}
\label{sec:num}

This section is devoted to the numerical validation of our theoretical results. The prototype model is the quantum system under the two-body interaction. Here we take the two-dimensional Morse potential as an example.
\begin{equation}\label{morse}
V(\bx) = - 2\me^{-\kappa(|\bx-\bx_A|- r_0)} + \me^{-2\kappa(|\bx-\bx_A|- r_0)}.
\end{equation}
In this case, $\bz(\bx) = (z_1, z_2) = \bx - \bx_A$ and $\psi(\bk)$ reads that
\begin{equation}
\psi(\bk) = \frac{1}{\mi \hbar} \left[  - \frac{2\kappa \me^{\kappa r_0}c_2 }{(|\bk|^2 + \kappa^2)^{3/2}} + \frac{2\kappa \me^{2\kappa r_0}c_2}{(|\bk|^2 + 4\kappa^2)^{3/2}} \right].
\end{equation}
Thus the pseudo-differential operator is given by that
\begin{equation}\label{pdo_2d_morse}
\begin{split}
\Theta_V[\varphi](\bx, \bk, t) =  &  - \frac{\kappa \me^{\kappa r_0} c_2}{\hbar} \int_{0}^{2\pi} \D \vartheta \int_{0}^{+\infty} \D r~ \frac{r \sin(2\bz(\bx) \cdot \bk^{\prime})}{\sqrt{r^2+(\kappa/2)^2}}  \frac{\Delta_{r\sigma}[\varphi](\bx, \bk, t)}{r^2 + (\kappa/2)^2}  \\
&+ \frac{\kappa \me^{2\kappa r_0} c_2}{\hbar} \int_{0}^{2\pi} \D \vartheta \int_{0}^{+\infty} \D r~ \frac{r \sin(2\bz(\bx) \cdot \bk^{\prime})}{\sqrt{r^2+\kappa^2}}  \frac{\Delta_{r\sigma}[\varphi](\bx, \bk, t)}{r^2 + \kappa^2} ,
\end{split}
\end{equation}
where $\bk^{\prime} = r\sigma = (r \cos \vartheta, r\sin \vartheta)$ and $c_2 = \Gamma(3/2)\pi^{-3/2} \approx 0.1592$. The stationary phase approximation to PDO reads
\begin{equation}
\begin{split}
\Theta^{\lambda_0}_V[\varphi](\bx, \bk, t)  = & \int_{B(\frac{\lambda_0}{|\bz(\bx)|})} \me^{\mi \bz(\bx) \cdot \bk^{\prime}} \psi(\bk^{\prime}) ( \varphi(\bx, \bk - \frac{\bk^{\prime}}{2}, t) - \varphi(\bx, \bk+\frac{\bk^{\prime}}{2}, t) ) \D \bk^{\prime} \\
&+2\int_{\frac{\lambda_0}{|\bz(\bx)|}}^{+\infty} \textup{Im}\left[\me^{\mi r |\bz(\bx)|} \left(\frac{2\pi}{\mi r |\bz(\bx)|}\right)^{\frac{1}{2}}\right] r{\psi}(r\sigma_+) \Delta_{r\sigma_+} [\varphi](\bx, \bk, t) \D r,
\end{split}
\end{equation}
where $\vartheta^+ =\textup{atan2}(z_2/z_1)$.

The variances can be monitored within the implementation of the Monte Carlo algorithm in \cite{ShaoXiong2019}. Here we adopt the Gaussian wavepacket of the form \eqref{def.Gaussian_wave_packet} as the initial {condition:}
\begin{equation}\label{def.Gaussian_wave_packet}
f_0(x_1, x_2, k_1, k_2) = \frac{1}{\pi^2}\me^{-0.5(x_1-8)^2 - 0.5(x_2-12)^2 - 2(k_1 - 0.5)^2 - 2(k_2 + 0.5)^2}.
\end{equation}
Other parameters are set as: $\bx_A = (10, 10), r_0 = 0.5, \kappa = 0.5$, $\hbar = m = 1$, 
$T=2$. We simulate the $10^5$ independent family trees and measure the $L^2$-error \cite{XiongShao2019}, which is proportional to the variance. The numerical solutions obtained by the highly accurate deterministic scheme are adopted as the reference \cite{XiongChenShao2016}. 
\begin{figure}[!h]
\centering
  \subfigure[Variance of WBRW-HJD.]{\includegraphics[width=0.64\textwidth,height=0.40\textwidth]{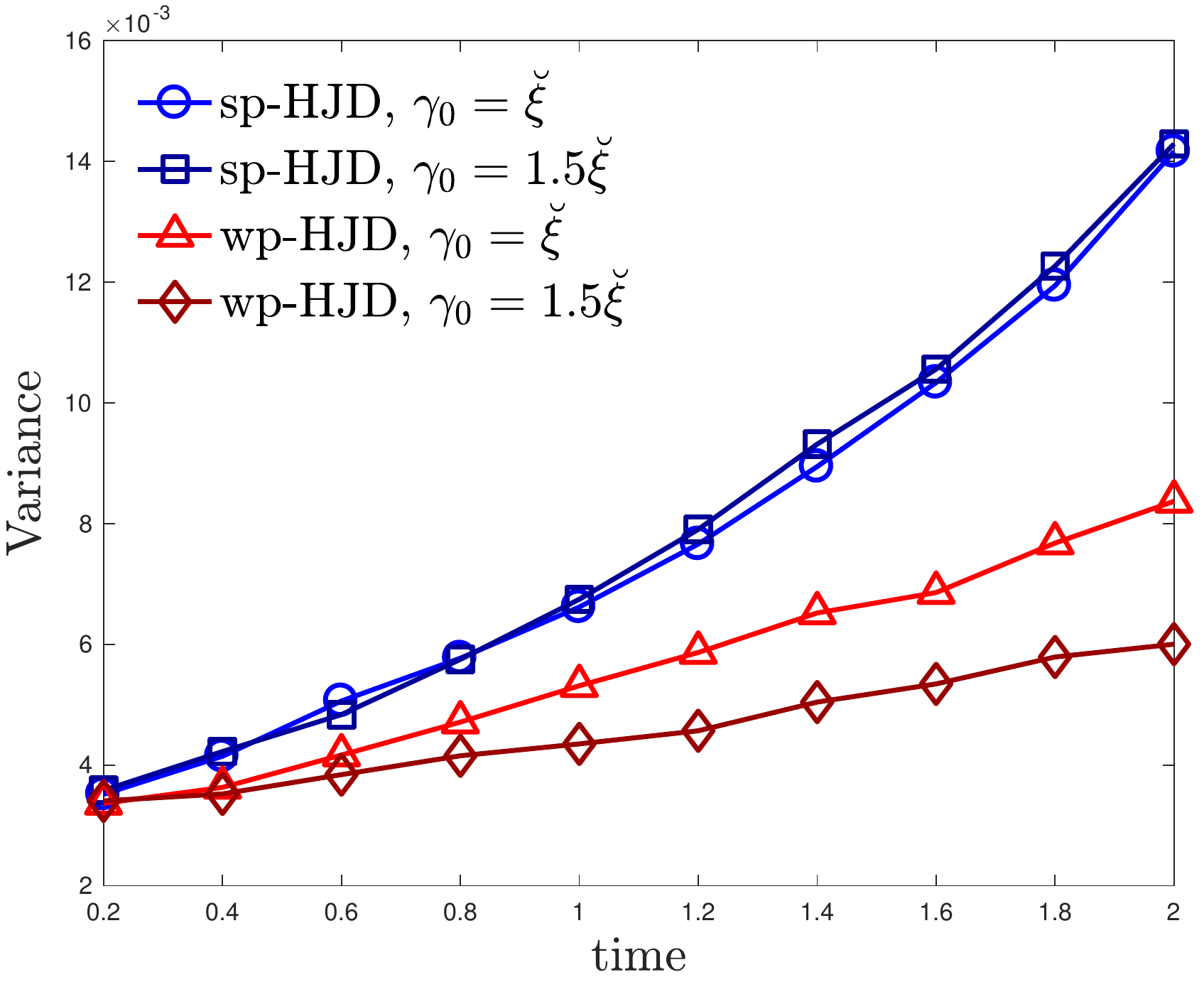} \label{fig_a}}\\
  \centering
  \subfigure[Variance of wp-SPA and sp-SPA.]{\includegraphics[width=0.64\textwidth,height=0.40\textwidth]{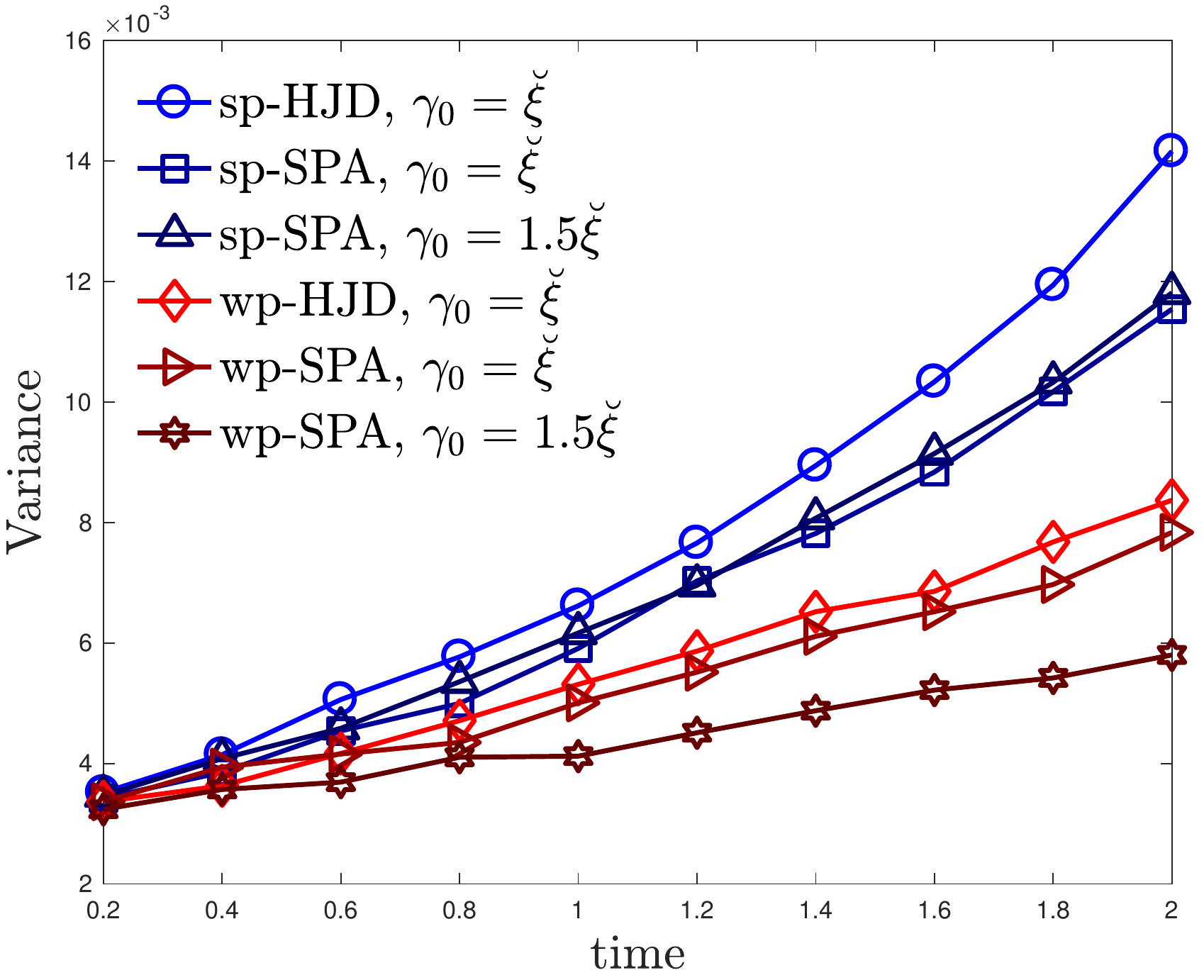} \label{fig_b}} \\
  \centering
  \subfigure[Variance of sp-SPA under different $\lambda_0$.]{\includegraphics[width=0.64\textwidth,height=0.40\textwidth]{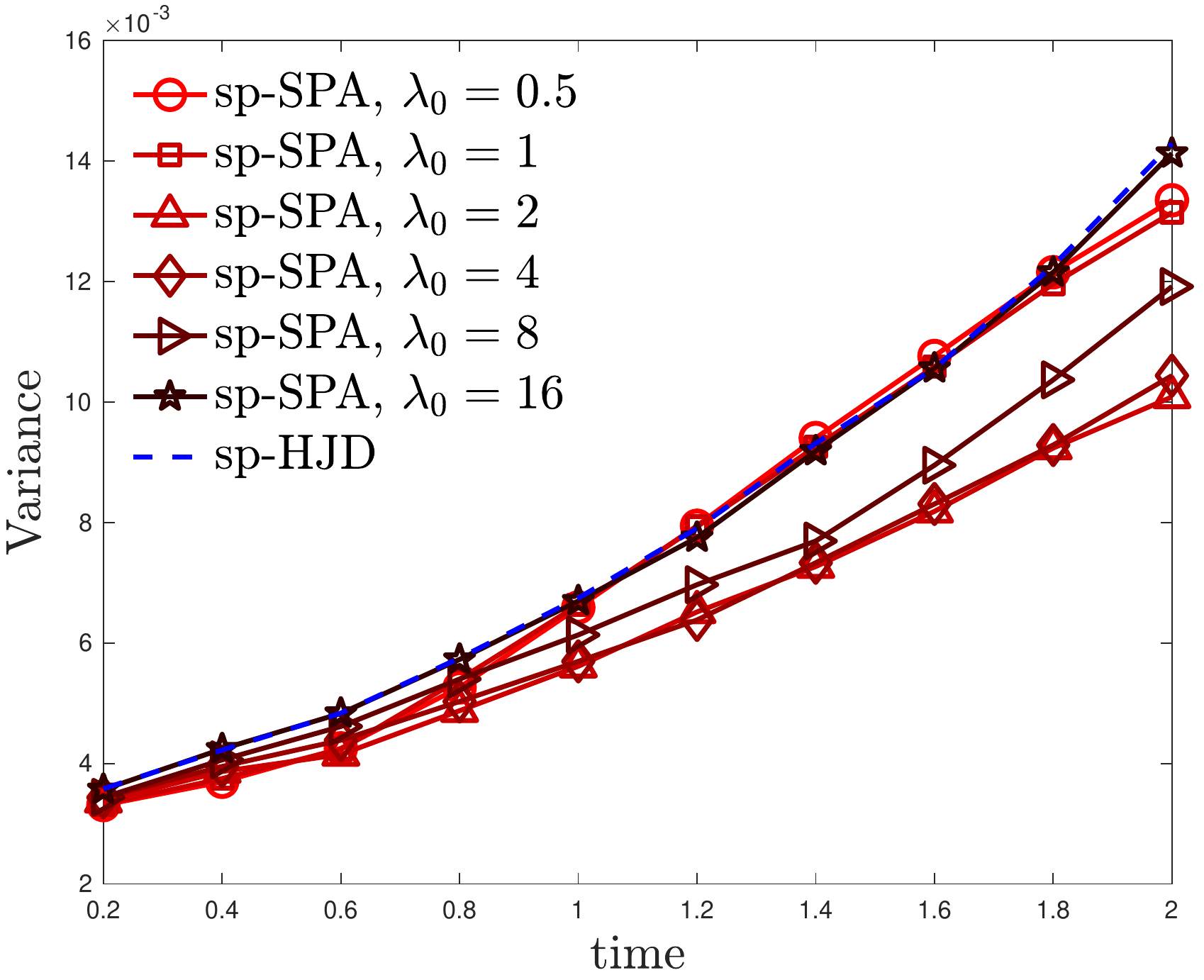} \label{fig_c}} 
  \caption{\small  The variances are monitored in the Monte Carlo simulations. The exponential growth of variances is observed. The choice of $\gamma_0$ has a great influence on the variance of both the weighted-particle WBRW-HJD (wp-HJD) and WBRW-SPA (wp-SPA), while it has little influence on that of the signed-particle counterparts (sp-HJD, sp-SPA). The variances are clearly reduced when the stationary phase approximation is adopted, but the results are not very satisfactory when $\lambda_0$ is either too large or too small.}
  \label{fig:variance}
\end{figure}

Fig.~\ref{fig_a} makes a comparison between the weighted-particle WBRW-HJD (wp-HJD) and the signed-particle one (sp-HJD). 
The variance of the weighted-particle model can be reduced by choosing a larger $\gamma_0$, while this does not hold for the signed-particle counterpart. Fig.~\ref{fig_b} compares the WBRW-HJD and WBRW-SPA when $\lambda_0 = 8$. It's readily seen that the variances of both weighted-particle WBRW-SPA (wp-SPA) and the signed-particle counterpart (sp-SPA)  are diminished when the stationary phase method is adopted,  while the variance of the signed-particle model is still independent of the choice of $\gamma_0$.  Finally, Fig.~\ref{fig_c} compares sp-SPA under different settings of $\lambda_0$. It is found that the reduction of variance might not be significant for too large $\lambda_0$ as $\alpha_\ast$ may be very close to 1, while  the errors in the asymptotic expansion become dominated for too small $\lambda_0$. A reasonable choice of $\lambda_0$ (such as $\lambda_0 = 2$) can strike a balance between the bias and variance. All of these observations perfectly coincide with our main theoretical results.

\section{Conclusion and discussion}
\label{sec:con}

In this paper, we have analyzed two classes of branching random walk (BRW) solutions to the Wigner (W) equation, including WBRW-HJD based on the Hahn-Jordan decomposition (HJD): $\Theta_V = \Theta^+_V - \Theta^-_V$, and WBRW-SPA based on the stationary phase approximation (SPA). The main idea is to split the nonlocal operator with anti-symmetric kernels into two parts and explain each of them as the generator of jump process of one branch of weighted particles. We have shown that although the first moment of WRBW-HJD recovers the solution of the Wigner equation, the $L^1$-bounds for the variances grows exponentially in time with the rate depending on the norm of $\Theta^\pm_V$, which is inconsistent with the decay rate of the pseudo-differential operator.  By contrast, the WBRW-SPA is able to capture the essential contributions from the localized parts and the variance of the resulting stochastic model can be diminished, at the cost of introducing a little bias.  These results are of great importance in applications, such as tackling a general form of nonlocal problems. In particular, it ameliorates the {numerical sign problem} in high dimensional {situation}, which may involves multiple pairs of potentials and high dimensional oscillatory integrals. Our ongoing work is to apply the WBRW-SPA to study the quantum dynamics under the Coulomb interaction, such as the Hydrogen atom.


\section*{Acknowledgement}
This research was supported by the National Natural Science Foundation of China (Nos.~11822102, 11421101) and High-performance Computing Platform of Peking University. 
SS is partially supported by Beijing Academy of Artificial Intelligence (BAAI). 
YX is partially supported by The Elite Program of Computational and Applied Mathematics for PhD Candidates in Peking University.


%

\end{document}